\documentclass[12pt,twoside]{amsart}

\usepackage{amsmath}
\usepackage{amsfonts}
\usepackage{amssymb}
\usepackage{amsthm}
\usepackage{color}
\usepackage{graphicx}
\usepackage[all]{xy}
\usepackage{fullpage}
\usepackage{url}

\theoremstyle{definition}
\newtheorem{defn}{Definition}[section]

\newtheorem{rmk}[defn]{Remark}
\newtheorem{question}[defn]{Question}
\theoremstyle{plain}
\newtheorem{thm}[defn]{Theorem}
\newtheorem{lem}[defn]{Lemma}
\newtheorem{prop}[defn]{Proposition}
\newtheorem{cor}[defn]{Corollary}

\newtheorem{conj}[defn]{Conjecture}

\def\C{\ensuremath{\mathbb{C}}}

\def\N{\ensuremath{\mathbb{N}}}
\def\P{\ensuremath{\mathbb{P}}}
\def\Q{\ensuremath{\mathbb{Q}}}
\def\R{\ensuremath{\mathbb{R}}}
\def\Z{\ensuremath{\mathbb{Z}}}

\def\AA{\ensuremath{\mathcal A}}
\def\BB{\ensuremath{\mathcal B}}
\def\CC{\ensuremath{\mathcal C}}
\def\DD{\ensuremath{\mathcal D}}
\def\EE{\ensuremath{\mathcal E}}
\def\FF{\ensuremath{\mathcal F}}

\def\HH{\ensuremath{\mathcal H}}
\def\II{\ensuremath{\mathcal I}}

\def\OO{\ensuremath{\mathcal O}}

\def\TT{\ensuremath{\mathcal T}}

\def\ch{\mathop{\mathrm{ch}}\nolimits}

\def\Coh{\mathop{\mathrm{Coh}}\nolimits}

\def\dim{\mathop{\mathrm{dim}}\nolimits}

\def\inf{\mathop{\mathrm{inf}}\nolimits}

\def\End{\mathop{\mathrm{End}}\nolimits}

\def\Ext{\mathop{\mathrm{Ext}}\nolimits}

\def\Hom{\mathop{\mathrm{Hom}}\nolimits}

\def\Ker{\mathop{\mathrm{Ker}}\nolimits}

\def\mod{\mathop{\mathrm{mod}}\nolimits}

\def\num{\mathop{\mathrm{num}}\nolimits}

\def\Pic{\mathop{\mathrm{Pic}}\nolimits}

\def\Rep{\mathop{\mathrm{Rep}}}
\def\rk{\mathop{\mathrm{rk}}}

\def\tilt{\mathop{\mathrm{tilt}}}

\def\td{\mathop{\mathrm{td}}\nolimits}

\def\Stab{\mathop{\mathrm{Stab}}\nolimits}

\def\into{\ensuremath{\hookrightarrow}}
\def\onto{\ensuremath{\twoheadrightarrow}}

\begin{document}

\title[Bridgeland Stability on Threefolds]{Bridgeland Stability on Threefolds -
Some Wall Crossings}

\author{Benjamin Schmidt}
\address{Department of Mathematics, The University of Texas at Austin, 2515 Speedway, Austin, TX 78712, USA}
\email{schmidt@math.utexas.edu}
\urladdr{https://sites.google.com/site/benjaminschmidtmath/}



\begin{abstract}
Following up on the construction of Bridgeland stability condition on $\P^3$ by Macr\`i, we develop techniques to study concrete wall crossing behavior for the first time on a threefold. In some cases, such as complete intersections of two hypersurfaces of the same degree or twisted cubics, we show that there are two chambers in the stability manifold where the moduli space is given by a smooth projective irreducible variety, respectively the Hilbert scheme. In the case of twisted cubics, we compute all walls and moduli spaces on a path between those two chambers. This allows us to give a new proof of the global structure of the main component, originally due to Ellingsrud, Piene, and Str{\o}mme. In between slope stability and Bridgeland stability there is the notion of tilt stability that is defined similarly to Bridgeland stability on surfaces. Beyond just $\P^3$, we develop tools to use computations in tilt stability to compute wall crossings in Bridgeland stability. 
\end{abstract}


\maketitle

\setcounter{tocdepth}{1}
\tableofcontents

\section{Introduction}

The introduction of stability condition on triangulated categories by Bridgeland in \cite{Bri07} has revolutionized the study of moduli spaces of sheaves on surfaces. We introduce techniques that worked on surfaces into the realm of threefolds. As an application, we deal with moduli spaces of sheaves in $\P^3$. It turns out that for certain classes $v \in K_0(\P^3)$, there is a chamber in the stability manifold $\Stab(\P^3)$, where the corresponding moduli space of stable objects with class $v$ is smooth, projective, and irreducible. The following theorem applies in particular to complete intersections of the same degree, twisted cubics, or the tangent bundle.

\begin{thm}[See Theorem \ref{thm:first_wall_general}]
Let $v = i \ch(\OO_{\P^3}(m)) - j \ch(\OO_{\P^3}(n))$ where $m,n \in \Z$ are integers with $n < m$ and $i,j \in \N$ are positive integers. Assume that $(v_0, v_1, v_2)$ is primitive in $\Z \oplus \Z \oplus \tfrac{1}{2}\Z$. There is a path $\gamma: [0,1] \to \Stab(\P^3)$ that satisfies the following properties.
\begin{enumerate}
  \item At $\gamma(0)$ the moduli space of semistable objects with class $v$ is empty.
  \item After the first wall on $\gamma$ the moduli space of semistable objects with class $v$ is smooth, irreducible, and projective.
  \item If $i \geq j$, then at $\gamma(1)$ the semistable objects with class $v$ are exactly slope stable $E \in \Coh(\P^3)$ with $\ch(E) = v$. In particular, there are no strictly semistable objects.
\end{enumerate}
\end{thm}

As an example, we compute all walls on the path of the last Theorem in the case of twisted cubics. Figure \ref{fig:bridgeland_walls_twisted} is a graphical representation of these walls.

\begin{thm}[See Theorem \ref{thm:bridgeland_twisted_cubics}]
Let $v = (1, 0, -3, 5) = \ch(\II_C)$ where $C \subset \P^3$ is a twisted cubic curve. There is a path $\gamma: [0,1] \to \Stab(\P^3)$ such that the moduli spaces of semistable objects with Chern character $v$ in its image outside of walls are given in the following order.
\begin{enumerate}
  \item The empty space $M_0 = \emptyset$.
  \item A smooth projective variety $M_1$ that contains ideal sheaves of twisted cubic curves as an open subset.
  \item A space with two components $M_2 \cup M'_2$ intersecting transversally. The space $M_2$ is a blow up of $M_1$ in a smooth locus. The intersection of the two components is precisely the exceptional locus of the blow up. It parametrizes plane singular cubic curves with a spatial embedded point at a singularity. The second component $M'_2$ is a $\P^9$-bundle over $\P^3 \times (\P^3)^{\vee}$. An open subset in $M'_2$ parametrizes plane cubic curves together with a potentially but not necessarily embedded point that is not scheme theoretically contained in the plane. 
  \item The Hilbert scheme of curves $C$ with $\ch(\II_C) = (1, 0, -3, 5)$. It is given as $M_2 \cup M'_3$, where $M'_3$ is a blow up of $M'_2$ in a smooth locus away from $M_2 \cap M_2'$. The exceptional locus parametrizes plane cubic curves together with a point scheme theoretically contained in the plane.
\end{enumerate}
\end{thm}

The Hilbert scheme of twisted cubics has been heavily studied. In \cite{PS85} it was shown that it has two smooth irreducible components of dimension $12$ and $15$ intersecting transversely in a locus of dimension $11$. In \cite{EPS87} it was shown that the closure of the space of twisted cubics in this Hilbert scheme is the blow up of another smooth projective variety in a smooth locus. While we still use the result in \cite{PS85}, we give a new proof of the description in \cite{EPS87}. Additionally, we also describe the second component in more detail. Based on the work in this article, it was recently shown how to remove the reliance on \cite{PS85} in \cite{Xia16}.

The literature on Hilbert schemes in projective space from a more classical point of view is vast. It turns out that the geometry of these spaces can be quite badly behaved. For example Mumford observed that there is an irreducible component in the Hilbert scheme in $\P^3$ containing smooth curves that is generically non reduced in \cite{Mum62}. However, Hartshorne proved that Hilbert schemes in projective space are at least connected in \cite{Har66}.

\subsection{Ingredients}

Bridgeland's original work was motivated by Calabi-Yau threefolds and related questions in physics. A fundamental issue in the theory of stability conditions on threefolds is the actual construction of Bridgeland stability conditions. A conjectural way was proposed in \cite{BMT14}, and has been proven for $\P^3$ in \cite{MacE14}, for the smooth quadric threefold in \cite{Sch14}, and for abelian threefolds in both \cite{MP13a, MP13b} and \cite{BMS14}. Tilt stability has been introduced in \cite{BMT14} as an intermediate notion between classical slope stability and Bridgeland stability on a smooth projective threefold $X$ over $\C$. The construction is analogous to Bridgeland stability on surfaces. The heart is a certain abelian category of two term complexes defined as a tilt of $\Coh(X)$, while the central charge is given by
$$Z^{\tilt}_{\alpha, \beta} = -H \cdot \ch_2^{\beta} + \frac{\alpha^2}{2} H^3 \cdot \ch_0^{\beta} + i H^2 \cdot \ch_1^{\beta},$$
where $H \in \Pic(X)$ is ample, $\alpha > 0$, $\beta \in \R$, and $\ch^{\beta} = e^{-\beta H} \cdot \ch$ is the twisted Chern character. More details on the construction is given in Section \ref{sec:construction}. Many techniques from the surfaces case still apply to tilt stability. Bayer, Macr\`i, and Toda propose that doing another tilt and choosing $s > 0$ will lead to a Bridgeland stability condition with central charge 
$$Z_{\alpha,\beta,s} = -\ch^{\beta}_3 + (s+\tfrac{1}{6})\alpha^2 H^2 \cdot \ch^{\beta}_1 + i (H \cdot \ch^{\beta}_2 - \frac{\alpha^2}{2} H^3 \cdot \ch^{\beta}_0).$$

While walls in tilt stability are well behaved and often computable, it is generally difficult to determine how a given moduli space changes at a wall. This issue arises from the fact that strictly semistable objects generally do not have Jordan-H\"older filtrations with unique stable factors in tilt stability. This is a non-issue in Bridgeland stability, but the price to pay is a hugely complicated structure of walls that is very difficult to handle. The following theorem attempts to resolve this catch-22 by relating the walls of both notions. It is one of the key ingredients for the two theorems above. If $v$ is the Chern character of an object in $D^b(X)$, we say that $(H^3 \cdot v_0, H^2 \cdot v_1, H \cdot v_2)$ is \emph{primitive} if it can not be written as $m(H^3 \cdot w_0, H^2 \cdot w_1, H \cdot w_2)$ for an integer $m \geq 2$ and the Chern character $w$ of a different object in $D^b(X)$.

\begin{thm}[See also Theorem \ref{thm:wall_intersecting_hyperbola}]
Assume that the construction of Bridgeland stability by Bayer, Macr\`i, and Toda works on $X$. Moreover, let $v$ be the Chern character of an object in $D^b(X)$ such that $(H^3 \cdot v_0, H^2 \cdot v_1, H \cdot v_2)$ is primitive. Then there are two paths $\gamma_1, \gamma_2 : [0,1] \to \Stab(X)$ with the following property. If $M$ is a moduli space of tilt stable objects for some $\alpha > 0$, $\beta \in \R$ that does not lie on any wall, then $M$ is a moduli space of Bridgeland stable objects along either $\gamma_1$ or $\gamma_2$.
\end{thm}

Note that the theorem does not preclude the existence of further chambers along these paths. In many cases, for example for twisted cubics, there are different exact sequences defining identical walls in tilt stability because the defining objects only differ in the third Chern character. However, by definition, changes in $\ch_3$ cannot be detected via tilt stability. In Bridgeland stability these walls often move apart and give rise to further chambers.

The computations in tilt stability in this article are very similar in nature to many computations about stability of sheaves on surfaces in \cite{ABCH13, BM14, CHW14, LZ13, MM13, Nue14, Woo13, YY14}. Despite the tremendous success in the surface case, the threefold case has barely been explored. Beyond the issue of constructing Bridgeland stability condition there are further problems that have made progress difficult.

\subsection{Further Questions}

For tilt stability parametrized by the $(\alpha, \beta)$ upper half-plane, there is at most one vertical wall, while all other walls are nested inside two piles of non intersecting semicircles. This structure is rather simple. However, in the case of Bridgeland stability on threefolds walls are given by real degree 4 equation. Already in the case of twisted cubics we can observe that they intersect (see Figure \ref{fig:bridgeland_walls_twisted}).

\begin{question}
Given a path $\gamma$ in the stability manifold and a class $v \in K_{\num}(X)$, is there a numerical criterion that determines all the walls on $\gamma$ with respect to $v$? If not, can we at least numerically restrict the amount of potential walls on $\gamma$ in an effective way?
\end{question}

We are only able to answer this question for the two paths described in Theorem \ref{thm:wall_intersecting_hyperbola}. The general situation seems to be more intricate. If we want to study stability in any meaningful way beyond tilt stability, we need at least partial answers to this question.

Another problem is the construction of reasonably behaved moduli spaces of Bridgeland semistable objects. A recent result by Piyaratne and Toda is a major step towards this.

\begin{thm}[{\cite{PT15}}]
Let $X$ be a smooth projective threefold such that the conjectural construction of Bridgeland stability from \cite{BMT14} works. Then any moduli space of semistable objects for such a Bridgeland stability condition is a universally closed algebraic stack of finite type over $\C$.
\end{thm}

If there are no strictly semistable objects, the moduli space becomes a proper algebraic space of finite type over $\C$. For certain applications such as birational geometry, we would like our moduli spaces to be projective.

\begin{question}
Assume $\sigma \in \Stab(X)$ is a Bridgeland stability condition and $v \in K_{\num}(X)$. When is the moduli space of $\sigma$-stable objects with class $v$ quasi-projective?
\end{question}

\subsection{Organization of the Article}

In Section \ref{sec:veryweak} we recall the notion of a very weak stability condition from \cite{BMS14} and \cite{PT15}. All our examples of stability conditions fall under this notion. Section \ref{sec:construction} describes the construction of both tilt stability and Bridgeland stability, and establishes some basic properties. In particular, we remark which techniques for Bridgeland stability on surfaces work without issues in tilt stability. In Section \ref{sec:p3} we deal with stability of line bundles or powers of line bundles on $\P^3$ by connecting these questions to moduli of quiver representations. Section \ref{sec:tilt_examples} deals with computing specific examples in $\P^3$ for tilt stability. Moreover, we discuss how many of those calculations can be handled by computer calculations. In Section \ref{sec:connection} we prove our main comparison theorem between Bridgeland stability and tilt stability. Finally, in Section \ref{sec:exmaples_bridgeland} we use this connection to finish the computations necessary to establish the two main theorems. 

\subsection{Notation}
\begin{center}
  \begin{tabular}{ r l }
    $X$ & smooth projective threefold over $\C$ \\
    $H$ & fixed ample divisor on $X$ \\
    $\II_{Z/X}$, $\II_Z$ & ideal sheaf of a closed subscheme $Z \subset X$ \\
    $D^b(X)$ & bounded derived category of coherent \\ & sheaves on $X$ \\
    $\ch(E)$ & Chern character of an object $E \in D^b(X)$  \\
    $\ch_{\leq l}(E)$ & $(\ch_{0}(E), \ldots,
    \ch_{l}(E))$ \\
    $H \cdot \ch(E)$ & $(H^3 \cdot \ch_{0}(E), H^2 \cdot \ch_{1}(E), H \cdot \ch_{2}(E), \ch_{3}(E))$ \\ & for an ample divisor $H$ on $X$ \\
    $H \cdot \ch_{\leq l}(E)$ & $(H^3 \cdot
    \ch_{0}(E), \ldots, H^{3-l} \cdot \ch_{l}(E))$ \\ & for an ample divisor $H$ on $X$ \\
    $K_{\num}(X)$ & the numerical Grothendieck group of $X$ \\
    $\Re(z)$ & real part of a complex number $z$ \\
    $\Im(z)$ & imaginary part of a complex number $z$
  \end{tabular}
\end{center}

\subsection*{Acknowledgements}
I would like to thank David Anderson, Arend Bayer, Patricio Gallardo, C\'esar Lozano Huerta, and Emanuele Macr\`i for insightful discussions or comments on this article. I also thank the referee for carefully reading the article and making many useful suggestions. I especially thank my advisor Emanuele Macr\`i for carefully reading preliminary versions of this article. Most of this work was done at the Ohio State University whose mathematics department was extraordinarily accommodating after my advisor moved. In particular, Thomas Kerler and Roman Nitze helped me a lot with handling the situation. Lastly, I would like to thank Northeastern University at which the finals details of this work were finished for their hospitality. The research was partially supported by NSF grants DMS-1160466 and DMS-1523496 (PI Emanuele Macr\`i) and a presidential fellowship of the Ohio State University.


\section{Very Weak Stability Conditions and the Support Property}
\label{sec:veryweak}

All forms of stability occurring in this article are encompassed by the
notion of a very weak stability condition introduced in Appendix B
of \cite{BMS14}. It will allow us to treat different forms of stability
uniformly. We will recall this notion more closely to how it was
defined in \cite{PT15}.

\begin{defn}
\label{def:heart}
A \textit{heart of a bounded t-structure} on $D^b(X)$ is a
full additive subcategory $\AA \subset D^b(X)$ such that
\begin{itemize}
  \item for integers $i > j$ and $A \in \AA[i]$, $B \in \AA[j]$ the vanishing
  $\Hom(A,B) = 0$ holds,
  \item for all $E \in \DD$ there are integers $k_1 > \ldots > k_m$ and a
  collection of triangles
  $$\xymatrix{
  0=E_0 \ar[r] & E_1 \ar[r] \ar[d] & E_2 \ar[r] \ar[d] & \ldots \ar[r] & E_{m-1}
  \ar[r] \ar[d] & E_m = E \ar[d] \\
  & A_1[k_1] \ar@{-->}[lu] & A_2[k_2] \ar@{-->}[lu] & & A_{m-1}[k_{m-1}]
  \ar@{-->}[lu] & A_m[k_m] \ar@{-->}[lu] }$$
  where $A_i \in \AA$.
\end{itemize}
\end{defn}

The heart of a bounded t-structure is automatically abelian. A proof of this
fact and a full introduction to the theory of t-structures can be found in
\cite{BBD82}. The standard example of a heart of a bounded t-structure on
$D^b(X)$ is given by $\Coh(X)$. More generally, whenever there is an equivalence $D^b(X) \cong D^b(\AA)$ for some abelian category $\AA$, then $\AA$ is the heart of a bounded t-structure on $D^b(X)$. While the converse does not hold in general, this example is key for the intuition behind t-structures.

\begin{defn}[\cite{Bri07}]
A \textit{slicing} of $D^b(X)$ is a collection of full additive subcategories $P(\phi) \subset
D^b(X)$ for all $\phi \in \mathbb{R}$ such that
\begin{itemize}
  \item $P(\phi)[1] = P(\phi + 1)$,
  \item if $\phi_1 > \phi_2$ and $A \in P(\phi_1)$, $B \in P(\phi_2)$ then
  $\Hom(A,B) = 0$,
  \item for all non-zero $E \in D^b(X)$ there are $\phi_1 > \ldots > \phi_m$ and a
  collection of triangles
  $$\xymatrix{
  0=E_0 \ar[r] & E_1 \ar[r] \ar[d] & E_2 \ar[r] \ar[d] & \ldots \ar[r] & E_{m-1}
  \ar[r] \ar[d] & E_m = E \ar[d] \\
  & A_1 \ar@{-->}[lu] & A_2 \ar@{-->}[lu] & & A_{m-1} \ar@{-->}[lu] & A_m
  \ar@{-->}[lu] }$$
  where $A_i \in P(\phi_i)$.
\end{itemize}
For this filtration of a non-zero element $E \in D^b(X)$, we write $\phi^-(E) := \phi_m$
and $\phi^+(E) := \phi_1$. Moreover, for $E \in P(\phi)$ we call $\phi(E):=\phi$
the \textit{phase} of $E$.
\end{defn}

The last property is called the \textit{Harder-Narasimhan filtration}. By setting
$\mathcal{A} := P((0,1])$ to be the extension closure of the subcategories $\{P(\phi):\phi \in (0,1]\}$ one gets the heart of a bounded t-structure from a slicing. In both cases of a slicing and the heart of a bounded t-structure it is not particularly difficult to show that the Harder-Narasimhan filtration is unique.

Let $v: K_0(X) \to \Gamma$ be a homomorphism where $\Gamma$ is a finite rank lattice. Fix $H$ to be an ample divisor on $X$. Then $v$ will usually be one of the homomorphisms $H \cdot \ch_{\leq l}$ defined by
$$E \mapsto (H^n \cdot \ch_0(E), \ldots, H^{n-l} \cdot \ch_l(E))$$ 
for some $l \leq n$.

\begin{defn}[\cite{PT15}]
A \textit{very weak pre-stability condition} on $D^b(X)$ is a pair $\sigma
= (P,Z)$, where $P$ is a slicing of $D^b(X)$ and $Z: \Gamma \to \mathbb{C}$ is a
homomorphism such that any non zero $E \in P(\phi)$ satisfies
$$Z(v(E)) \in \begin{cases}
\R_{>0} e^{i \pi \phi} &\text{for } \phi \in \R \backslash \Z\\
\R_{\geq 0} e^{i \pi \phi} &\text{for } \phi \in \Z.
\end{cases}$$
\end{defn}

This definition is short and good for abstract argumentation, but it is not very practical for defining concrete examples. As before, the heart of a bounded t-structure can be defined by $\mathcal{A} := P((0,1])$. The usual way to define a very weak pre-stability condition is to instead define the heart of a bounded t-structure $\mathcal{A}$ and a central charge $Z: \Gamma \to \C$ such that $Z \circ v$ maps $\mathcal{A} \backslash \{ 0 \}$ to the upper half plane or the non positive real line $\{re^{i\pi\varphi} : r\geq0, \varphi \in (0,1]\}$. The subcategory $P(\phi)$ for $\phi \in (0,1]$ consists of all semistable objects $E$ such that
$$Z(v(E)) \in \begin{cases}
\R_{>0} e^{i \pi \phi} &\text{for } \phi \in \R \backslash \Z \\
\R_{\geq 0} e^{i \pi \phi} &\text{for } \phi \in \Z.
\end{cases}$$
More precisely, we can define a slope function by
$$\mu_{\sigma} := -\frac{\Re(Z)}{\Im(Z)},$$
where dividing by $0$ is interpreted as $+\infty$. Then an object $E \in \AA$ is
called \textit{(semi-)stable} if for all monomorphisms $A \into E$ in $\AA$ we
have $\mu_{\sigma}(A) < (\leq) \mu_{\sigma}(A/E)$. More generally, an element $E
\in D^b(X)$ is called (semi-)stable if there is $m \in \Z$ such that $E[m] \in \AA$ is
(semi-)stable. A semistable but not stable object is called \textit{strictly
semistable}. Moreover, one needs to show that Harder-Narasimhan filtrations
exist inside $\AA$ with respect to the slope function $\mu_{\sigma}$ to actually
get a very weak pre-stability condition. We interchangeably use $(\AA, Z)$ and $(P,Z)$
to denote the same very weak pre-stability condition.

An important tool is the support property. It was introduced in
\cite{KS08} for Bridgeland stability conditions, but can be adapted without much
trouble to very weak stability conditions (see \cite[Section 2]{PT15}). We also
recommend \cite[Appendix A]{BMS14} for a nicely written treatment of this
notion. Without loss of generality we can assume that if $E \in \AA$ and $Z(v(E)) = 0$, then
$v(E) = 0$. If not we replace $\Gamma$ by a suitable quotient.

\begin{defn}
A very weak pre-stability condition $\sigma = (\AA, Z)$ satisfies the
\textit{support property} if there is a bilinear form $Q$ on $\Gamma \otimes \R$
such that
\begin{enumerate}
  \item all semistable objects $E\in \AA$ satisfy the inequality $Q(v(E),v(E))
  \geq 0$ and
  \item all non zero vectors $v \in \Gamma \otimes \R$ with $Z(v) = 0$ satisfy $Q(v, v) < 
  0$.
\end{enumerate}
A very weak pre-stability condition satisfying the support property is called a
\textit{very weak stability condition}.
\end{defn}

By abuse of notation, we will write $Q(E,F)$ instead of $Q(v(E), v(F))$ for $E, F \in D^b(X)$. We will also use the notation $Q(E) = Q(E,E)$.

Let $\Stab^{vw}(X,v)$ be the set of very weak stability conditions on $X$ with
respect to $v$. This set can be given a topology as the coarsest topology such
that the maps $(\AA,Z) \mapsto Z$, $(\AA, Z) \mapsto \phi^+(E)$, and $(\AA, Z)
\mapsto \phi^-(E)$ for any $E \in D^b(X)$ are continuous.

\begin{lem}[{\cite{BMS14}[Section 8, Lemma A.7 \& Proposition A.8]}]
\label{lem:Qzero}
Assume that $Q$ has signature $(2, \rk \Gamma - 2)$ and $U$ is a path connected open subset of $\Stab^{vw}(X,v)$ such that all $\sigma \in U$ satisfy the support property with respect to $Q$.
\begin{itemize}
  \item  If $E \in D^b(X)$ with $Q(E) = 0$ is $\sigma$-stable for some $\sigma
  \in U$ then it is $\sigma'$-stable for all $\sigma' \in U$ unless it is
  destabilized by an object $F$ with $v(F)=0$.
  \item Let $\rho$ be a ray in $\C$ starting at the origin. Then
  $$\CC^+ = Z^{-1}(\rho) \cap \{ Q \geq 0 \}$$
  is a convex cone for any very weak stability condition $(\AA,Z) \in U$.
  \item Moreover, any vector $w \in \CC^+$ with $Q(w) = 0$ generates an extremal
  ray of $\CC^+$.
\end{itemize}
\end{lem}

Only the situation of an actual stability condition is handled in \cite{BMS14}.
In that situation there are no objects $F$ in the heart with $v(F) = 0$.
However, exactly the same arguments go through in the case of a very weak
stability condition.

\begin{defn}
A \textit{numerical wall} inside $\Stab^{vw}(X,v)$ (or a subspace of it) with
respect to an element $w \in \Gamma$ is a proper non trivial solution set of an
equation $\mu_{\sigma}(w) = \mu_{\sigma}(u)$ for a vector $u \in \Gamma$.

A subset of a numerical wall is called an actual \textit{wall}, if for each point of the subset there is an exact sequence of semistable objects $0 \to F \to E \to G \to 0$ in $\AA$, where $v(E) = w$ and $\mu_{\sigma}(F) = \mu_{\sigma}(G)$ numerically defines the wall.
\end{defn}

Walls in the space of very weak stability conditions satisfy certain numerical restrictions with respect to $Q$.

\begin{lem}
\label{lem:discriminant_properties}
Let $\sigma = (\AA, Z)$ be a very weak stability condition satisfying the
support property with respect to $Q$ (it is actually enough for $Q$ to be
negative semi-definite on $\Ker Z$).
\begin{enumerate}
  \item Let $F,G \in \AA$ be semistable objects. If $\mu_{\sigma}(F) = \mu_{\sigma}(G)$, then $Q(F,G) \geq 0$.
  \item Assume there is an actual wall defined by an exact sequence $0 \to F \to
  E \to G \to 0$. Then $0 \leq Q(F) + Q(G) \leq Q(E)$.
\end{enumerate}
\begin{proof}
We start with the first statement. If $Z(F) = 0$ or $Z(G) = 0$, then $Q(F,G) = 0$. If not, there is $\lambda > 0$ such that $Z(F - \lambda G) = 0$. Therefore, we get
$$0 \geq Q(F - \lambda G) = Q(F) + \lambda^2 Q(G) - 2\lambda Q(F,G).$$
The inequalities  $Q(F) \geq 0$ and $Q(G) \geq 0$ lead to $Q(F,G) \geq 0$.
For the second statement we have
$$Q(E) = Q(F) + Q(G) + 2Q(F,G) \geq 0.$$
Since all four terms are positive, the claim follows.
\end{proof}
\end{lem}

\begin{rmk}
\label{rmk:bigger_lattice}
Since $Q$ has to be only negative semi-definite on $\Ker Z$ for the lemma to apply, it is sometimes possible to define $Q$ on a bigger lattice than $\Gamma$. For example, we will define a very weak stability condition factoring through $v = H \cdot \ch_{\leq 2}$, but apply the lemma for $v = H \cdot \ch$, where everything is still well defined later on.
\end{rmk}

The most well known example of a very weak stability condition is slope stability. We will slightly generalize it for notational purposes. Let $H$
be a fixed ample divisor on $X$. Moreover, pick a real number $\beta$. Then the
twisted Chern character $\ch^{\beta}$ is defined to be $e^{-\beta H} \cdot \ch$. In
more detail, one has

\begin{align*}
\ch^{\beta}_0 &= \ch_0, \\
\ch^{\beta}_1 &= \ch_1 - \beta H \cdot \ch_0 ,\\
\ch^{\beta}_2 &= \ch_2 - \beta H \cdot \ch_1 + \frac{\beta^2}{2} H^2 \cdot \ch_0, \\
\ch^{\beta}_3 &= \ch_3 - \beta H \cdot \ch_2 + \frac{\beta^2}{2} H^2 \cdot \ch_1 -
\frac{\beta^3}{6} H^3 \cdot \ch_0.
\end{align*}

In this case $v = H \cdot \ch_{\leq 1}$. The central charge is
given by
$$Z^{sl}_{\beta}(r,c) = -(c - \beta r) + ir.$$
The heart of a bounded t-structure in this case is simply $\Coh(X)$. The existence of Harder-Narasimhan filtrations was first proven for curves in \cite{HN74}, but holds in general. Finally the support property is satisfied for $Q = 0$. We will denote the corresponding slope function by
$$\mu_{\beta} := \frac{H^2 \cdot \ch^{\beta}_1}{H^3 \cdot \ch^{\beta}_0} = \frac{H^2\cdot \ch_1}{H^3 \cdot \ch_0} - \beta.$$
Note that the modification by $\beta$ does not change stability itself but just shifts the value of the slope.


\section{Constructions and Basic Properties}
\label{sec:construction}

\subsection{Tilt Stability}

In \cite{BMT14} the notion of tilt stability was introduced as an
auxiliary notion in between classical slope stability and Bridgeland stability on threefolds. We will recall its construction, and prove a few properties. From now on let $\dim X = 3$.

The process of tilting is used to obtain a new heart of a bounded t-structure.
For more information on the general theory of tilting we refer to \cite{HRS96}.
A torsion pair is defined by
\begin{align*}
\TT_{\beta} &= \{E \in \Coh(X) : \text{any quotient $E \onto G$
satisfies $\mu_{\beta}(G) > 0$} \}, \\
\FF_{\beta} &=  \{E \in \Coh(X) : \text{any subsheaf $F \subset E$
satisfies $\mu_{\beta}(F) \leq 0$} \}.
\end{align*}
A new heart of a bounded t-structure is defined as the extension
closure $\Coh^{\beta}(X) = \langle \FF_{\beta}[1],
\TT_{\beta} \rangle$. In this case $v = H \cdot \ch_{\leq 2}$. Let
$\alpha > 0$ be a positive real number. The central charge is given by
$$Z^{\tilt}_{\alpha, \beta}(r,c,d) = -(d - \beta c +
\frac{\beta^2}{2} r) + \frac{\alpha^2}{2}r + i(c - \beta r).$$
The corresponding slope function is
$$\nu_{\alpha, \beta} = \frac{H \cdot \ch^{\beta}_2 - \frac{\alpha^2}{2} H^3
\cdot \ch^{\beta}_0}{H^2 \cdot \ch^{\beta}_1}.$$
Note that in regard to \cite{BMT14} this slope has been modified by switching
$\omega$ with $\sqrt{3} \omega$. We prefer this point of view for aesthetical
reasons because it will make the walls semicircles and not just ellipses. Every
object in $\Coh^{\beta}(X)$ has a Harder-Narasimhan filtration due to
\cite[Lemma 3.2.4]{BMT14}. The support property is directly linked to the
Bogomolov inequality. This inequality was first proven for
slope semistable sheaves in \cite{Bog78}. We define the bilinear form by
$Q^{\tilt}((r,c,d), (R,C,D)) = Cc - Rd - Dr$.

\begin{thm}[{Bogomolov Inequality for Tilt Stability,
\cite[Corollary 7.3.2]{BMT14}}]
Any $\nu_{\alpha, \beta}$-semistable object $E \in \Coh^{\beta}(X)$ satisfies
\begin{align*}
Q^{\tilt}(E) &= (H^2 \cdot \ch_1^{\beta}(E))^2 - 2(H^3 \cdot \ch_0^{\beta})(H \cdot \ch_2^{\beta}) \\
&= (H^2 \cdot \ch_1(E))^2 - 2(H^3 \cdot \ch_0)(H \cdot \ch_2) \geq 0.
\end{align*}
\end{thm}

As a consequence $(\Coh^{\beta}, Z^{\tilt}_{\alpha, \beta})$
satisfies the support property with respect to $Q^{\tilt}$. On smooth
projective surfaces this is already enough to get a Bridgeland stability condition (see
\cite{Bri08, AB13}). On threefolds this notion is not able to properly handle
geometry that occurs in codimension three as we will see.

\begin{prop}[{\cite[Appendix B]{BMS14}}]
\label{prop:tilt_continuous}
The function $\R_{>0} \times \R \to \Stab^{vw}(X,v)$ defined by $(\alpha, \beta) \mapsto (\Coh^{\beta}(X), Z^{\tilt}_{\alpha, \beta})$ is continuous. Moreover, walls with respect to a class $w \in \Gamma$ in the image of this map are locally finite.
\end{prop}

Numerical walls in tilt stability satisfy Bertram's Nested Wall Theorem. For surfaces it was proven in \cite{MacA14}.

\begin{thm}[Structure Theorem for Walls in Tilt Stability]
\label{thm:structure_theorem}
Fix a vector $(R,C,D) \in \Z^2 \times 1/2 \Z$. All numerical walls in
the following statements are with respect to $(R,C,D)$.
\begin{enumerate}
  \item Numerical walls in tilt stability are of the form 
  $$x\alpha^2 + x\beta^2 + y\beta + z = 0$$
  for $x = Rc - Cr$, $y = 2(Dr - Rd)$ and $z = 2(Cd - Dc)$.
  In particular, they are either semicircles with center on the $\beta$-axis
  or vertical rays.
  \item If two numerical walls given by $\nu_{\alpha, \beta}(r,c,d) =
  \nu_{\alpha, \beta}(R,C,D)$ and $\nu_{\alpha, \beta}(r',c',d') = \nu_{\alpha,
  \beta}(R,C,D)$ intersect for any $\alpha \geq 0$ and $\beta \in \R$ then $(r,c,d)$, $(r',c',d')$
  and $(R,C,D)$ are linearly dependent. In particular, the two walls are
  completely identical.
  \item The curve $\nu_{\alpha, \beta}(R,C,D) = 0$ is given by the hyperbola
  $$R\alpha^2 - R\beta^2 + 2C\beta - 2D = 0.$$
  Moreover, this hyperbola intersect all semicircles at their top point.
  \item If $R \neq 0$, there is exactly one vertical numerical wall given by
  $\beta = C/R$. If $R = 0$, there is no vertical wall.
  \item If a numerical wall has a single point at which it is an actual wall,
  then all of it is an actual wall.
\end{enumerate}
\begin{proof}
Part (1) and (3) are straightforward but lengthy computations only relying on
the numerical data.

A wall can also be described as two vectors mapping to the same
line under the homomorphism $Z^{\tilt}_{\alpha, \beta}$. This homomorphism
maps surjectively onto $\C$. Therefore, at most two linearly independent vectors
can be mapped onto the same line. That proves (2).

In order to prove (4), observe that a vertical wall occurs when $x = 0$ holds.
By the above formula for $x$ this implies
$$c = \frac{Cr}{R}$$
in case $R \neq 0$. A direct computation shows that the equation simplifies to
$\beta = C/R$. If $R = 0$ and $C \neq 0$, then $r = 0$. This implies that the
two slopes are the same for all or no $(\alpha, \beta)$. If $R = C
= 0$, then all objects with this Chern character are automatically semistable
and there are no walls at all.

Let $0 \to F \to E \to G \to 0$ be an exact sequence of tilt semistable objects
in $\Coh^{\beta}(X)$ that defines an actual wall. If there is a point on the
numerical wall at which this sequence does not define a wall anymore, then there are two possibilities. If either $F$, $E$, or $G$ destabilize at a point along the numerical wall, then that would mean two of its numerical walls intersect in contradiction to (2). The other possibility is that one of the objects $F$, $E$, or $G$ leave the category $\Coh^{\beta}(X)$. As long as this object stays semistable this can only happen along its numerical vertical wall. Again two numerical walls intersect in contradiction to (2).
\end{proof}
\end{thm}

A generalized Bogomolov inequality involving third Chern characters for
tilt semistable objects with $\nu_{\alpha, \beta} = 0$ has been conjectured in
\cite{BMT14}. In \cite{BMS14} it was shown that the conjecture is equivalent to
the following more general inequality that drops the hypothesis $\nu_{\alpha, \beta} = 0$.

\begin{conj}[BMT Inequality]
Any $\nu_{\alpha, \beta}$-semistable object $E \in \Coh^{\beta}(X)$ satisfies
$$\alpha^2 Q^{\tilt}(E) + 4(H \cdot \ch_2^{\beta}(E))^2 - 6(H^2 \cdot \ch_1^{\beta}) \ch_3^{\beta} \geq 0.$$
\end{conj}

By using the definition of $\ch^{\beta}(E)$ and expanding the expression one can
find $x,y \in \R$ depending on $E$ such that the inequality becomes
$$\alpha^2 Q^{\tilt}(E) + \beta^2 Q^{\tilt}(E) + x\beta + y \geq 0.$$
This means the solution set is given by the complement of a semi-disc with
center on the $\beta$-axis or a quadrant to one side of a vertical line.
The conjecture is known for $\P^3$ \cite{MacE14}, the smooth quadric threefold \cite{Sch14}, and all abelian threefolds \cite{BMS14, MP13a, MP13b}.

Another question that comes up in concrete situations is the question whether a given tilt semistable object is a sheaf. For a fixed $\beta$ let
$$c = \inf \{H^2 \cdot \ch_1^{\beta}(E) > 0 : E \in \Coh^{\beta}(X)\}.$$

\begin{lem}[{\cite[Lemma 7.2.1 and 7.2.2]{BMT14}}]
\label{lem:large_volume_limit_tilt}
An object $E \in \Coh^{\beta}(X)$ that is $\nu_{\alpha, \beta}$-semistable
for all $\alpha \gg 0$ is given by one of three possibilities.
\begin{enumerate}
  \item $E = H^0(E)$ is a pure sheaf supported in dimension greater than or equal
  to two that is slope semistable.
  \item $E = H^0(E)$ is a sheaf supported in dimension less than or equal to one.
  \item $H^{-1}(E)$ is a torsion free slope semistable sheaf and $H^0(E)$ is
  supported in dimension less than or equal to one. Moreover, if $\mu_{\beta}(E) <
  0$ then $\Hom(F, E) = 0$ for all sheaves $F$ of dimension less than or equal
  to one.
\end{enumerate}
An object $F \in \Coh^{\beta}(X)$ with $H^2 \cdot \ch_1^{\beta} \in \{0, c\}$ is
$\nu_{\alpha, \beta}$-semistable if and only if it is given by one of the
three types above. 
\end{lem}

\begin{rmk}
The lemma implies the following useful fact that will be used multiple times in the upcoming sections. Let $E \in \Coh^{\beta}(X)$ for fixed $\beta$ with $H^2 \cdot \ch_1^{\beta}(E) = c$. Then $E$ is either $\nu_{\alpha, \beta}$-semistable for all $\alpha > 0$ or for no $\alpha > 0$. Indeed, by definition of $\Coh^{\beta}(X)$ any potentially destabilizing subobject $F$ satisfies either $H^2 \cdot \ch^{\beta}_1(F) = 0$ or $H^2 \cdot \ch^{\beta}_1(F) = c$. In the second case the quotient $E/F$ satisfies
$H^2 \cdot \ch^{\beta}_1(E/F) = 0$. Therefore, either the quotient or the
subobject has infinite slope, while $E$ does not have infinite slope.
\end{rmk}


Using the same proof as in the surface case in \cite[Proposition 14.1]{Bri08} leads to the following lemma.

\begin{lem}
\label{lem:stable_sheaves_stable_in_tilt}
Assume $E \in \Coh(X)$ is a slope stable sheaf and $\beta < \mu(E)$. Then $E$ is
$\nu_{\alpha, \beta}$-stable for all $\alpha \gg 0$.
\end{lem}

\subsection{Bridgeland Stability}

We will recall the definition of a Bridgeland stability condition from
\cite{Bri07} and show how they can be conjecturally constructed on threefolds
based on the BMT-inequality as described in \cite{BMT14}. It is known that the
inequality holds on $\P^3$ due to \cite{MacE14} and we will apply it in a later
section to study concrete examples of moduli spaces of complexes in this case.

\begin{defn}
A \textit{Bridgeland (pre-)stability condition} on the category $D^b(X)$ is a very weak (pre-)stability condition $(P, Z)$ such that $Z(E) \neq 0$ for all semistable
objects $E \in D^b(X)$. We denote the subspace of Bridgeland
stability conditions by $\Stab(X,v) \subset \Stab^{vw}(X,v)$.
\end{defn}

If $\AA = P((0,1])$ is the corresponding heart, then we could have equivalently defined a Bridgeland stability condition by the property $Z(E) \neq 0$ for all non zero $E \in \AA$. Note that in this situation choosing the heart to be $P((0,1])$ instead of $P((\phi - 1,\phi])$ for any $\phi \in \mathbb{R}$ is arbitrary and any other choice works just as well. In some very special cases it is possible to choose $\phi$ such that the corresponding heart is equivalent to the category of In order to have any hope of actually computing wall-crossing
behavior, it is necessary for walls in Bridgeland stability to be somewhat
reasonably behaved. The following result due to \cite[Section 9]{Bri08} is a major step towards that.representations of a quiver with relations. This will be particularly useful in the case of $\P^3$.

\begin{thm}[{\cite[Section 7]{Bri07}}]
The map $(\AA,Z) \mapsto Z$ from $\Stab(X,v)$ to $\Hom(\Gamma, \C)$ is a local
homeomorphism. In particular, $\Stab(X,v)$ is a complex manifold.
\end{thm}

The \emph{mass} of $E \in D^b(X)$ with respect to $\sigma = (\AA, Z) \in \Stab(X, v)$ is defined as
\[
m_{\sigma}(E) = \sum_i |Z(A_i)|,
\]
where the $A_i$ are the semistable factors of $E$ in its Harder-Narasimhan filtration. In order to have any hope of actually computing wall-crossing behavior, it is necessary for walls in Bridgeland stability to be somewhat reasonably behaved. The following result due to \cite[Section 9]{Bri08} is a major step towards that.

\begin{thm}
\label{thm:locally_finite}
Let $S \subset D^b(X)$ be a set of objects of bounded mass in a compact subset $\BB \subset \Stab(X, v)$, i.e.,
\[
\sup\{m_{\sigma}(E) : E \in S \} < \infty
\]
for some $\sigma \in \BB$. Then the subset of semistable objects of $S$ varies in a finite wall and chamber structure in $\BB$.
\end{thm}

The first application of this theorem is that walls in $\Stab(X,v)$ for objects with fixed numerical invariants are locally finite. More widely, it also implies that the Harder-Narasimhan filtrations of objects with fixed numerical invariants vary in a locally finite wall and chamber structure. Lastly, even the the stable factors of these semistable factors vary in a locally finite wall and chamber structure. The last two statements are a consequence of the fact that the mass of a semistable or stable factor $A$ of $E$ is smaller than or equal to the mass of $E$.

An important question is how moduli spaces change set theoretically at walls. In case the destabilizing subobject and quotient are both stable this has a satisfactory answer due to \cite[Lemma 5.9]{BM11}. Note that this proof does not work in the case of very weak stability conditions due to the lack of unique factors in the Jordan-H\"older filtration.

\begin{lem}
\label{lem:wall_crossing}
Let $\sigma = (\AA, Z) \in \Stab(X)$ such that there are stable object $F,G \in
\AA$ with $\mu_{\sigma}(F) = \mu_{\sigma}(G)$. Then there is an open
neighborhood $U$ around $\sigma$ where non trivial extensions $0 \to F \to E
\to G \to 0$ are stable exactly for those $\sigma' \in U$ with $\phi_{\sigma'}(F) <
\phi_{\sigma'}(G)$.
\begin{proof}
Since stability is an open property, there is an open neighborhood $U$ of
$\sigma$ in which both $F$ and $G$ are stable. By Theorem \ref{thm:locally_finite}, we can shrink $U$ to obtain a neighborhood of $\sigma$ in which all walls for $v(F) + v(G)$ intersect $\sigma$, and there are only finitely many such walls. Even more, we can choose $U$ such that the same holds for all stable factors of objects with invariants $v(F) + v(G)$.

In particular, if $\sigma' \in U$ and $A \subset E$ is a stable subobject with larger or equal slope at $\sigma'$, then $A$ must have larger or equal slope at $\sigma$ as well. However, $E$ is strictly $\sigma$-semistable, and therefore, $A$ is a Jordan-H\"older factor $E$. Since Jordan-H\"older factors are unique up to order, we must have $A \cong F$ or $A \cong G$. However, since $E$ is obtained from a non-trivial extension, we have $\Hom(G, E) = 0$. Therefore, $A = F$ and the claim follows.
\end{proof}
\end{lem}

It turns out that while constructing very weak stability conditions is not very
difficult, constructing Bridgeland stability conditions is in general a wide
open problem. Note that for any smooth projective variety of dimension bigger
than or equal to two, there is no Bridgeland stability condition factoring
through the Chern character for $\AA = \Coh(X)$ due to \cite[Lemma
2.7]{Tod09}.

Tilt stability does not define Bridgeland stability as can be seen by the fact that
skyscraper sheaves are mapped to the origin. In \cite{BMT14} it was conjectured
that one has to tilt $\Coh^{\beta}(X)$ again as follows in order to construct a
Bridgeland stability condition on a threefold. Let 
\begin{align*}
\TT_{\alpha, \beta} &= \{E \in \Coh^{\beta}(X) : \text{any quotient $E
\onto G$ satisfies $\nu_{\alpha, \beta}(G) > 0$} \}, \\
\FF_{\alpha, \beta} &=  \{E \in \Coh^{\beta}(X) : \text{any subobject $F
\into E$ satisfies $\nu_{\alpha, \beta}(F) \leq 0$} \}
\end{align*}
and set $\AA^{\alpha, \beta}(X) = \langle \FF_{\alpha, \beta}[1],
\TT_{\alpha, \beta} \rangle $. For any $s>0$ they define
\begin{align*}
Z_{\alpha,\beta,s} &= -\ch^{\beta}_3 + (s+\tfrac{1}{6})\alpha^2
H^2 \cdot \ch^{\beta}_1 + i (H \cdot \ch^{\beta}_2 - \frac{\alpha^2}{2} H^3 \cdot \ch^{\beta}_0), \\
\lambda_{\alpha,\beta,s} &= -\frac{\Re(Z_{\alpha,\beta,s})}{\Im(Z_{\alpha,\beta,s})}.
\end{align*}

In this case the bilinear form is given by
\begin{align*}
Q_{\alpha, \beta, K}((r,c,d,e),(R,C,D,E)) = \
&Q^{\tilt}((r,c,d),(R,C,D)) (K\alpha^2 + \beta^2) \\
&+ (3Er + 3Re - Cd - Dc) \beta \\
&- 3Ce - 3Ec + 4Dd.
\end{align*}
for some $K \in (1, 6s + 1)$. Notice that for $K=1$ this comes directly from the BMT-inequality.

\begin{thm}[{\cite[Corollary 5.2.4]{BMT14}, \cite[Lemma 8.8]{BMS14}}]
If the BMT inequality holds, then $(\AA^{\alpha, \beta}(X), Z_{\alpha, \beta,
s})$ is a Bridgeland stability condition for all $s > 0$. The support
property is satisfied with respect to $Q_{\alpha, \beta, K}$ for any $K \in (1, 6s+1)$.
\end{thm}

Note that as a consequence the BMT inequality holds for all $\lambda_{\alpha,
\beta, s}$-stable objects. In \cite[Proposition 8.10]{BMS14} it is shown
that this implies a continuity result just as in the case of tilt stability.

\begin{prop}
\label{prop:bridgeland_continuous}
The function $\R_{>0} \times \R \times \R_{>0} \to \Stab(X,v)$ defined by 
$(\alpha, \beta, s) \mapsto (\AA^{\alpha, \beta}(X), Z_{\alpha, \beta, s})$ is
continuous.
\end{prop}

In the case of tilt stability we have seen that the limiting stability for
$\alpha \to \infty$ is closely related with slope stability. The first step in
connecting Bridgeland stability with tilt stability is a similar result. For an
object $E \in \AA^{\alpha, \beta}(X)$ we denote the \textit{cohomology with respect to
the heart $\Coh^{\beta}(X)$} by $\HH_{\beta}^i(E)$. It is defined by the
property that $\HH_{\beta}^i(E)[i] \in \Coh^{\beta}(X)$ is a factor in the
Harder-Narasimhan filtration of $E$.

\begin{lem}[{\cite[Lemma 8.9]{BMS14}}]
\label{lem:large_volume_limit_bridgeland}
If $E \in \AA^{\alpha, \beta}(X)$ is $Z_{\alpha, \beta, s}$-semistable for all
$s \gg 0$, then one of the following two conditions holds.
\begin{enumerate}
  \item $E = \HH_{\beta}^0(E)$ is a $\nu_{\alpha, \beta}$-semistable object.
  \item $\HH_{\beta}^{-1}(E)$ is $\nu_{\alpha, \beta}$-semistable and
  $\HH_{\beta}^0(E)$ is a sheaf supported in dimension $0$.
\end{enumerate}
\end{lem}


\section{Stability on $\P^3$}
\label{sec:p3}

In the case of $\P^3$ more can be proven than in the general case. In this
section the connection to stability of quiver representations will be
recalled and a stability result about line bundles will be proven. It was
already shown in \cite{BMT14} that a line bundle $L$ is tilt stable if $Q^{\tilt}(L) = 0$. This condition always holds in Picard rank $1$. However, we need a slightly more refined result that holds in the special case of $\P^3$.

\begin{prop}
\label{prop:line_bundles_uniquely_stable}
Let $v = \pm \ch(\OO(n)^{\oplus m})$ for integers $n,m$ with $m > 0$. Then for any $\alpha > 0$, $\beta \in \R$, and $s > 0$ the object $\OO(n)^{\oplus m}$, or a homological shift of it, is the unique tilt semistable and
Bridgeland semistable object with Chern character $\pm v$. Moreover, in the case $m = 1$ the line bundle $\OO(n)$ is stable.
\end{prop}

For the proof we will need a connection between Bridgeland stability and quiver
representations. We will recall exceptional collections after
\cite{Bon90}.

\begin{defn}
\begin{enumerate}
  \item An object $E \in D^b(X)$ is called an \textit{exceptional object} if
  $\Ext^l(E,E) = 0$ for all $l \neq 0$ and $\Hom(E,E) = \mathbb{C}$.
  \item A sequence $E_0, \ldots, E_n \in D^b(X)$ of exceptional objects is a
  \textit{full exceptional collection} if $\Ext^l(E_i, E_j) = 0$ for all $l$ and
  $i > j$ and $D^b(X) = \langle E_0, \ldots, E_n \rangle$, i.e., $D^b(X)$ is
  generated by $E_0, \ldots, E_n$ through shifts and extensions.
  \item A full exceptional collection $E_0, \ldots, E_n$ is called
  \textit{strong} if additionally $\Ext^l(E_i, E_j) = 0$ for all $l \neq 0$ and
  $i < j$.
\end{enumerate}
\end{defn}

\begin{thm}[\cite{Bon90}]
\label{thm:Bondal_equivalence}
Let $E_0, \ldots, E_n$ be a strong full exceptional collection on $D^b(X)$, $E = \bigoplus E_i$,
$A = \End(E)$, and $\mod-A$ be the category of right
$A$-modules of finite rank. Then the functor
$$R\Hom(E, \cdot): D^b(X) \to D^b(\mod-A)$$
is an exact equivalence. Under this identification the $E_i$ correspond to
the indecomposable projective $A$-modules.
\end{thm}

In particular, the category $\mod-A$ becomes the heart of a bounded t-structure on
$D^b(X)$ with this identification. In the case of $\P^3$ this heart can be connected
to some stability conditions. In the following statement $T$ is the tangent bundle on $\P^3$.

\begin{thm}[\cite{MacE14}]
\label{thm:macri_p3}
If $\alpha < 1/3$ and $\beta \in (-2/3, 0]$, then
$$\mathcal{C} := \langle \mathcal{O}(-1)[3], T(-2)[2], \mathcal{O}[1],
\mathcal{O}(1) \rangle = P_{\alpha, \beta}((\phi, \phi + 1])$$
for some $\phi \in (0,1)$ and the Bridgeland stability condition $(P_{\alpha,
\beta}, Z_{\alpha, \beta, s})$ for small enough $s > 0$.
Moreover, $\mathcal{C}$ is the category $\mod-A$ for some finite dimensional
algebra $A$ coming from an exceptional collection as in Theorem
\ref{thm:Bondal_equivalence}. The four objects generating $\mathcal{C}$
correspond to the simple representations.
\end{thm}

We will only require the following corollary of this statement.

\begin{cor}
\label{cor:semistable_powers_structure_sheaf}
An object $E \in \AA^{\alpha, \beta}(\P^3)$ with $\ch(E) = \pm (m, 0, 0, 0)$ is $\lambda_{\alpha, \beta, s}$-semistable for $\alpha < 1/3$, $\beta \in (-2/3, 0]$, and $s > 0$ small enough if and only $E \cong \OO^{\oplus m}$ or $\cong \OO^{\oplus m}[1]$.
\end{cor}

\begin{proof}
Let $\CC$ be the category defined in Theorem \ref{thm:macri_p3}. Then $E \in \AA^{\alpha, \beta}(\P^3)$ being $\lambda_{\alpha, \beta, s}$-semistable implies that $E \in \CC$ or $E[1] \in \CC$. By definition of $\CC$, there are four integers $a,b,c,d$, all non-negative or all non-positive, such that 
\[
\ch(E) = a \ch(\OO(-1)[3]) + b \ch(T(-2)[2]) + c \ch(\OO[1]) + d \ch(\OO(1)).
\]
A straightforward computation shows that $a = b = d = 0$ and $c = \pm m$ is the only possibility. Since $\OO(-1)[3]$, $T(-2)[2]$, $\OO[1]$ and $\OO(1)$ are the simple object in the finite length category $\CC$, we must have $E \cong \OO^{\oplus m}$ or $E \cong \OO^{\oplus m}[1]$.

Vice versa, $\OO^{\oplus m}[1]$ is a direct sum of simple objects in $\CC$. This means it has to be semistable irregardless of the stability condition.
\end{proof}

\begin{proof}[Proof of Proposition \ref{prop:line_bundles_uniquely_stable}]
By definition $\ch^{\beta + a}(E \otimes \OO(a)) = \ch^{\beta}(E)$ for any $a \in \Z$. It follows directly from the definitions that an object $E$ is $\nu_{\alpha, \beta}$-semistable if and only if $E \otimes \OO(a)$ is $\nu_{\alpha, \beta + a}$-semistable. Similarly, $E$ is $\lambda_{\alpha, \beta, s}$-semistable if and only if $E \otimes \OO(a)$ is $\lambda_{\alpha, \beta + a, s}$-semistable. In particular, we can tensor with $\OO(-n)$ to reduce the statement to the case $n = 0$, and $v = \pm (m,0,0,0)$.

Corollary \ref{cor:semistable_powers_structure_sheaf} implies the statement for some $\alpha$, $\beta$, $s$ in Bridgeland stability. Next, we will extend this to all $\alpha$, $\beta$, $s$ in Bridgeland stability. Notice that
$Q_{\alpha, \beta, K}(v) = 0$. By Lemma \ref{lem:Qzero} the object $\OO$ is
Bridgeland stable for all $\alpha$, $\beta$, $s$. Let $E \in \AA^{\alpha, \beta}(\P^3)$
be $Z_{\alpha, \beta, s}$-semistable with $\ch(E) = v$. By Lemma
\ref{lem:Qzero}, the class $v$ spans an extremal ray of the cone $\CC^+ =
Z^{-1}_{\alpha, \beta, s}(\R_{\geq 0} v) \cap \{ Q_{\alpha, \beta, K} \geq 0
\}$. In particular, that means all its Jordan-H\"older factors are scalar
multiples of $v$. If $m=1$, then $v$ is primitive.
Therefore, $E$ is actually stable and then $E$ is also stable for $\alpha =
\tfrac{1}{4}$ and $\beta = 0$, i.e., $E$ is $\OO$ or a shift of it. Assume $m >
1$. Since there are no stable objects with class $v$ at $\alpha = \tfrac{1}{4}$
and $\beta = 0$, Lemma \ref{lem:Qzero} implies that $E$ is strictly semistable.
Therefore, the case $m=1$ implies that all the Jordan-H\"older factors are $\OO$.

The next step is to show semistability of $\OO^{\oplus m}$ in tilt stability. For this,
we just need deal with $m=1$. We have $Q^{\tilt}(\OO) = 0$. By Lemma
\ref{lem:Qzero} we know that $\OO$ is tilt stable everywhere or nowhere unless
it is destabilized by an object supported in dimension $0$. In that case $\beta
= 0$ is a wall. However, that cannot happen since there are no morphism from or
to $\OO[1]$ for any skyscraper sheaf. 
Since $v$ is primitive, semistability of $\OO$ is equivalent to stability. For $\beta = 0$ and $\alpha \gg 0$ we know that $\OO$ is semistable due to Lemma \ref{lem:large_volume_limit_tilt}.

Now we will show that any tilt semistable object $E$ with $\ch(E) = v$ has to be
$\OO^{\oplus m}$ for $\alpha = 1$, $\beta = -1$. We have $\nu_{1, -1}(E) = 0$. Therefore,
$E[1]$ is in the category $\AA^{1, -1}(\P^3)$. The Bridgeland slope is
$\lambda_{1,-1,s}(E[1]) = \infty$ independently of $s$. This means $E$ is
Bridgeland semistable and by the previous argument $E \cong \OO^{\oplus m}$.

We will use $Q^{\tilt}(v) = 0$ and Lemma \ref{lem:Qzero} similarly as in the
Bridgeland stability case to extend it to all of tilt stability. We start with the case $\beta < 0$. Let $E \in \Coh^{\beta}(\P^3)$ be a tilt semistable
object with $\ch(E) = v$. By using Lemma \ref{lem:Qzero}, the class $v$ spans
an extremal ray of the cone $\CC^+ = (Z_{\alpha, \beta}^{\tilt})^{-1}(\R_{\geq
0} v) \cap \{ Q^{\tilt} \geq 0 \}$. In particular, that means all its stable factors have
Chern character $(1,0,0,e)$. The BMT inequality shows $e \leq 0$. But since all
the stable factors add up to $v$, this means $e = 0$. Therefore, we reduced to the
case $m = 1$. In this case Lemma \ref{lem:Qzero} does the job as before.

If $\beta = 0$, the situation is more involved, since skyscraper sheaves can be
stable factors. All stable factors have Chern characters
of the form $(-1,0,0,e)$ or $(0,0,0,f)$. In this case $f \geq 0$. Let $F$ be
such a stable factor with Chern character $(-1,0,0,e)$. By openness of stability
$F$ is stable in a whole neighborhood that includes points with $\beta < 0$ and
$\beta > 0$. The BMT-inequality in both cases together implies $e = 0$. But then
$f = 0$ follows from the fact that Chern characters are additive. Again we
reduced to the case $m = 1$. By openness of stability and the result for $\beta
< 0$ we are done with this case. The case $\beta > 0$ can now be handled in the
same way as $\beta < 0$ by using Lemma \ref{lem:Qzero} again.
\end{proof}

In the case of tilt stability there is an even stronger statement. If
$\beta > n$, we do not need to fix $\ch_3$ to get the same conclusion.

\begin{prop}
\label{prop:line_bundles_tilt_slope_positive}
Let $v = -\ch_{\leq 2}(\OO(n)^{\oplus m})$ for integers $n,m$ with $m > 0$.
Then $\OO(n)^{\oplus m}[1]$ is the unique tilt semistable object $E$ with
$\ch_{\leq 2}(E) = v$ for any $\alpha > 0$ and $\beta > n$.
\begin{proof}
The semistability of $\OO(n)^{\oplus m}[1]$ has already been shown in
Proposition \ref{prop:line_bundles_uniquely_stable}. As in the previous proof,
we can use tensoring by $\OO(-n)$ to reduce to the case $n = 0$. This means $v =
(m, 0, 0)$.

Let $E \in \Coh^{\beta}(\P^3)$ be a tilt stable object for some $\alpha > 0$ and
$\beta > 0$ with $\ch(E) = (-m, 0, 0, e)$. The BMT-inequality implies $e \geq
0$. Since $Q^{\tilt}(E) = 0$, we can use Lemma \ref{lem:Qzero} to get that $E$ is
tilt stable for all $\beta > 0$. If $E$ is also stable for $\beta = 0$, then
using the BMT-inequality for $\beta < 0$ implies $e = 0$. Assume $E$ becomes
strictly semistable at $\beta = 0$. By Lemma \ref{lem:Qzero} the class $v$ spans
an extremal ray of the cone $\CC^+ = (Z_{\alpha, \beta}^{\tilt})^{-1}(\R_{\geq
0} v) \cap \{ Q^{\tilt} \geq 0 \}$. That means all stable factors must have
Chern characters of the form $(-m',0,0,e')$ for some $0 \leq m' \leq m$. If $m'
\neq 0$ then using the BMT-inequality for both $\beta < 0$ and $\beta > 0$
implies $e' = 0$. If $m' = 0$, then $e' > 0$. However, all the third Chern
characters add up to the non positive number $e$. This is only possible if $e = e' = 0$
and no stable factor has $m'=0$. By Proposition
\ref{prop:line_bundles_uniquely_stable} this means $E \cong \OO[1]^m$ and since
$E$ is stable this is only possible if $m=1$.

Let $E \in \Coh^{\beta}(\P^3)$ be a strictly tilt semistable object for some
$\alpha > 0$ and $\beta > 0$ with $\ch_{\leq 2}(E) = (-m, 0, 0)$. Since
$Q^{\tilt}(E) = 0$, we can use Lemma \ref{lem:Qzero} again to get that all
stable factors $F$ have $\ch_{\leq 2}(F) = (-m', 0, 0)$ for some $m' > 0$. By the previous part of the proof this means $m' = 1$ and $F \cong \OO[1]$ finishes the proof.
\end{proof}
\end{prop}

Note that a version of this proposition is already known even without assuming $X = \P^3$. However, the proof in our case is much simpler. In the proof of \cite[Proposition 3.12]{BMS14} it is shown that the numerical condition $Q^{\tilt}(E) = 0$ together with the fact that $E$ is $\nu_{\alpha, \beta}$-semistable for $\beta > n$ implies that $E$ is the homological shift of a vector bundle. They use previous results from \cite{LM16}. From there the classical result \cite[Theorem 2]{Sim92} can finish the proof.

We end this section by recalling a basic characterization of ideal sheaves in
$\P^k$.

\begin{lem}
\label{lem:torsion_free_rank_one}
Let $E \in \Coh(\P^k)$ be torsion free of rank one and $\ch_1(E) = 0$.
Then either $E \cong \OO$ or there is a subscheme $Z \subset \P^k$ of
codimension at least two such that $E \cong \II_Z$.
\begin{proof}
We have the inclusion $E \into E^{\vee \vee}$. The sheaf $E^{\vee \vee}$ is
reflexive of rank one, i.e., locally free (see \cite[Chapter 1]{Har80} for basic
properties of reflexive sheaves). Due to $\ch_1(E) = 0$ and $\rk(E) = 1$, we get
$ E^{\vee \vee} \cong \OO$. Therefore, either $E \cong \OO$ or there is a
subscheme $Z \subset \P^k$ such that $E \cong \II_Z$. If $Z$ were not of
codimension at least two, then $c_1(E) \neq 0$.
\end{proof}
\end{lem}


\section{Examples in Tilt Stability}
\label{sec:tilt_examples}

In examples, techniques from the last two sections can be used to determine
walls in tilt stability. This is similar to work on surfaces as done in various
articles (\cite{ABCH13, BM14, CHW14, LZ13, MM13, Nue14, Woo13, YY14}). We will showcase this for some
cases in $\P^3$. For any $v \in K_{\num}(X)$ we denote the set
of tilt semistable objects with Chern character $\pm v$ for some $\alpha > 0$
and $\beta \in \R$ by $M^{\tilt}_{\alpha, \beta}(v)$.

\subsection{Certain Sheaves}

Let $m,n \in \Z$ be integers with $n < m$ and $i,j \in \N$ positive integers. We
define a class $v = i \ch(\OO_{\P^3}(m)) - j \ch(\OO_{\P^3}(n))$. In this section we study
walls for this class $v$ in tilt stability. Interesting examples of sheaves with
this Chern character are ideal sheaves of complete intersections of two surfaces
of the same degree, ideal sheaves of twisted cubics, or the tangent bundle. In this generality we
will determine the smallest wall in tilt stability on one side of the vertical
wall.

\begin{thm}
\label{thm:smallestWall}
A wall not containing any smaller wall in tilt-stability for objects with class $v$ is
given by the equation $\alpha^2 + (\beta - \tfrac{m+n}{2})^2 =
(\tfrac{m-n}{2})^2$. All semistable objects $E$ at the wall are given by
extensions of the form $0 \to \OO(m)^{\oplus i} \to E \to \OO(n)^{\oplus j}[1]
\to 0$. Moreover, there are no tilt semistable objects of class $v$ inside this semicircle.
\begin{proof}
The semicircle defined by $Q_{\alpha, \beta, 1}(v) = 0$ coincides
with the wall claimed to exist. Therefore, the BMT-inequality implies that no
smaller semicircle can be a wall. Moreover, Proposition
\ref{prop:line_bundles_uniquely_stable} shows that both $\OO(m)^{\oplus i}$ and
$\OO(n)^{\oplus j}[1]$ are tilt semistable. The equation
$\nu_{\alpha, \beta}(\OO(m)) = \nu_{\alpha, \beta}(\OO(n))$ is exactly the
equation $\alpha^2 + (\beta - \tfrac{m+n}{2})^2 = (\tfrac{m-n}{2})^2$.
Therefore, we are left to prove the second assertion.

Let $F$ be a stable factor of $E$ at the wall. By Lemma \ref{lem:discriminant_properties} and Remark \ref{rmk:bigger_lattice} we get $Q_{\alpha, \beta, 1}(F) = 0$ at the wall. Since $F$ is stable, it is stable in a whole neighborhood around
the wall. But $Q_{\alpha, \beta, 1}(F)$ will be negative on one side of the wall
unless $Q_{\alpha, \beta, 1}(F) = 0$ for all $\alpha$, $\beta$. Taking the limit
$\alpha \to \infty$ implies $Q^{\tilt}(F) = 0$.

Assume that $\ch(F) = (r,c,d,e)$. Then $Q^{\tilt}(F) = 0$ implies $c^2 - 2rd = 0$. If $r = 0$, then $c = 0$. That cannot happen, because the wall would be a vertical line and not a semicircle in that situation. Thus, we can assume $r \neq 0$. In particular, the equality $d = \tfrac{c^2}{2r}$ holds. The point $\alpha_0 = \tfrac{m-n}{2}$, $\beta_0 = \tfrac{m+n}{2}$ lies on the wall. Since $F$ and $E$ have the same slope at $(\alpha_0, \beta_0)$, a straightforward but lengthy computation shows $c = mr$ or $c = nr$. That means $\ch(F)$ is a multiple of the Chern character of either $\OO(m)$ or $\OO(n)$. Since $F$ was assumed to be stable, Proposition \ref{prop:line_bundles_uniquely_stable} shows that $F$ has to be one of those line bundles.

Since the Chern characters of these two lines bundles are linearly independent,
we know that any decomposition of $E$ into stable factors must contain $i$ times
$\OO(m)$ and $j$ times $\OO(n)[1]$. The proof can be finished by the fact that
$\Ext^1(\OO(m), \OO(n)[1]) = 0$.
\end{proof}
\end{thm}

In the case of the Chern character of an ideal sheaf of a curve there is also a bound on the biggest wall.

\begin{prop}
\label{prop:biggestWall}
Let $v = (1,0,-d,e)$ be the Chern character of an ideal sheaf of a curve of
degree $d$. The biggest wall for $M^{\tilt}_{\alpha, \beta}(v)$ and
$\beta < 0$ is contained inside the semicircle defined by $\nu_{\alpha,
\beta}(v) = \nu_{\alpha, \beta}(\OO(-1))$. The biggest wall in the case
$\beta > 0$ is contained inside the semicircle defined by $\nu_{\alpha,
\beta}(v) = \nu_{\alpha, \beta}(\OO(1))$.
\begin{proof}
We start by showing there is no wall intersecting $\beta = \pm 1$. Let $E$ be
tilt semistable for $\beta = \pm 1$ and some $\alpha$ with $\ch(E) = \pm v$.
Then $\ch^{\pm 1}_1(E) = 1$ holds. If $E$ is strictly tilt semistable, then
there is an exact sequence $0 \to F \to E \to G \to 0$ of tilt semistable
objects with the same slope. However, either $\ch^{\pm 1}(F) = 0$ or $\ch^{\pm
1}(G) = 0$, a contradiction. The numerical wall $\nu_{\alpha, \beta}(v) =
\nu_{\alpha, \beta}(\OO(\pm 1))$ contains the point $\alpha = 0$, $\beta = \pm
1$. The argument is finished by the fact that numerical walls cannot intersect.
\end{proof}
\end{prop}

\subsection{Twisted Cubics}

While describing all the walls in general seems to be hard, we can handle the situation in examples. Let $C$ be a twisted cubic curve in $\P^3$. We will compute all the walls in tilt stability for $\beta < 0$ for the class $\ch(\II_C)$. There is a locally free resolution $0 \to \OO(-3)^{\oplus 2} \to \OO(-2)^{\oplus 3} \to \II_C \to 0$. This leads to
$$\ch^{\beta}(\II_C) = \left(1, -\beta, \frac{\beta^2}{2} - 3,
-\frac{\beta^3}{6} + 3\beta + 5\right).$$

\begin{figure}[h!]
        \centering
        \includegraphics[width=0.5\textwidth]{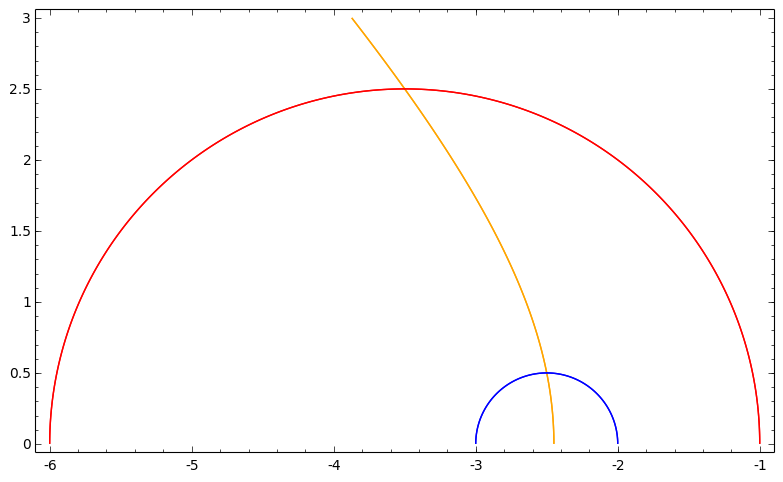}
        \caption{Walls in tilt stability}\label{fig:tilt_walls_twisted}
\end{figure}

\begin{thm}
\label{thm:tiltStabilityTwistedCubics}
There are two walls for $M^{\tilt}_{\alpha, \beta}(1,0,-3,5)$ for $\alpha > 0$
and $\beta < 0$. Moreover, the following table lists pairs of tilt semistable
objects whose extensions completely describe all strictly semistable objects
at each of the corresponding walls. Let $V$ be a plane in $\P^3$, $P \in \P^3$
and $Q \in V$.
\begin{center}
  \begin{tabular}{ r | l }
    & \\
    $\alpha^2 + (\beta + \frac{5}{2})^2 = \left(\frac{1}{2}\right)^2$ &
    $\OO(-2)^{\oplus 3}$, $\OO(-3)[1]^{\oplus 2}$ \\
    & \\
    \hline
    & \\
    $\alpha^2 + (\beta + \frac{7}{2})^2 = \left(\frac{5}{2}\right)^2$ &
    $\II_P(-1)$, $\OO_V(-3)$ \\ 
    & $\OO(-1)$, $\II_{Q/V}(-3)$\\
    & \\
  \end{tabular}
\end{center}
The hyperbola $\nu_{\alpha, \beta}(1,0,-3) = 0$ is given by the equation
$$\beta^2 - \alpha^2 = 6.$$
\end{thm}

In order to prove the theorem we need to put numerical restrictions on
potentially destabilizing objects. We do this in a series of lemmas.

\begin{lem}
\label{lem:idealSheafRestriction}
Fix $\beta \in \Z$ and let $E \in \Coh^{\beta}(\P^3)$ be $\nu_{\alpha, \beta}$-semistable for some $\alpha > 0$.
\begin{enumerate}
  \item If $\ch^{\beta}(E) = (1,1,d,e)$, then $d - 1/2 \in \Z_{\leq 0}$.
  In the case $d = -1/2$, we get $E \cong \II_L(\beta + 1)$ where $L \subset \P^3$ is a line
  plus $1/6-e$ (possibly embedded) points. If $d = 1/2$, then
  $E \cong \II_Z(\beta + 1)$ for a zero dimensional subscheme $Z \subset \P^3$
  of length $1/6 - e$.
  \item If $\ch^{\beta}(E) = (0,1,d,e)$, then $d + 1/2 \in \Z$ and $E
  \cong I_{Z/V}(\beta + d + 1/2)$ where $Z$ is a dimension zero subscheme of
  length $1/24 + d^2/2 - e$.
\end{enumerate}
\begin{proof}
Lemma \ref{lem:large_volume_limit_tilt} implies $E$ to be either a
torsion free sheaf or a pure sheaf supported in dimension $2$. By tensoring $E$
with $\OO(-\beta)$ we can reduce to the case $\beta = 0$.

In case (1) we have $\ch(E \otimes \OO(-1)) = (1, 0, d - 1/2, 1/3 - d + e)$.
Lemma \ref{lem:torsion_free_rank_one} implies that $E \otimes \OO(-1)$ is an
ideal sheaf of a subscheme $Z \subset \P^3$. This implies $d - 1/2 \in \Z_{\leq 0}$.
If $d = 1/2$, then $Z$ is zero dimensional of length $d - e - 1/3 = 1/6 -
e$. In case $d = -1/2$, the subscheme $Z$ is a line plus points. The Chern
character of the ideal sheaf of a line is given by $(1,0,-1,1)$. Therefore, the
number of points is $1 + d - e - 1/3 = 1/6 - e$.

In case (2) $E$ is supported on a plane $V$. We will use Lemma
\ref{lem:torsion_free_rank_one} on $V$. In order to so, we need to use
the Grothendieck-Riemann-Roch Theorem to compute the Chern character of $E$ on $V$. The Todd classes of
$\P^2$ and $\P^3$ are given by $\td(\P^2) = (1, \tfrac{3}{2}, 1)$ and $\td(\P^3)
= (1, 2, \tfrac{11}{6}, 1)$. Therefore, we get
\begin{align*}
i_* \left(\ch_V(E) \cdot \left(1, \frac{3}{2}, 1\right) \right) &= \left(0, 1, d, e
\right) \cdot \left(1,2,\frac{11}{6},1\right) \\
&= \left(0, 1, d + 2, 2d + e + \frac{11}{6}\right)
\end{align*}
where $i: V \into \P^3$ is the inclusion. Thus, we have $\ch_V(E) = (1, d
+ 1/2, d/2 + e + 1/12)$ and $d + 1/2$ is indeed an integer. Moreover, we can
compute $$\ch_V(E \otimes \OO(-d-1/2)) = (1, 0, e - \frac{d^2}{2} - \frac{1}{24}).$$
Using Lemma \ref{lem:torsion_free_rank_one} on $V$ concludes the proof.
\end{proof}
\end{lem}

The next lemma determines the Chern characters of possibly destabilizing objects
for $\beta = -2$.

\begin{lem}
\label{lem:chAtMinusTwo}
If an exact sequence $0 \to F \to E \to G \to 0$ in $\Coh^{-2}(\P^3)$ defines a
wall for $\beta = -2$ with $\ch_{\leq 2}(E) = (1, 0, -3)$ then up to
interchanging $F$ and $G$ we have $\ch^{-2}_{\leq 2}(F) = (1, 1, \tfrac{1}{2})$
and $\ch^{-2}_{\leq 2}(G) = (0, 1, -\tfrac{3}{2})$.
\begin{proof}
The argument is completely independent of $F$ being a quotient or a subobject. 
We have $\ch^{-2}_{\leq 2}(E) = (1, 2, -1)$.

Let $\ch^{-2}_{\leq 2}(F) = (r,c,d)$. By definition of $\Coh^{-2}(\P^3)$, we
have $0 \leq c \leq 2$. If $c=0$, then $\nu_{\alpha, -2}(F) = \infty$ and this
is in fact no wall for any $\alpha > 0$. If $c=2$, then the same argument for
the quotient $G$ shows there is no wall. Therefore, $c=1$ must hold. We can
compute
\begin{align*}
\nu_{\alpha, -2}(E) = - \frac{2 + \alpha^2}{4}, \ \nu_{\alpha,
-2}(F) = d - \frac{r \alpha^2}{2}.
\end{align*}
The wall is defined by $\nu_{\alpha, -2}(E) = \nu_{\alpha, -2}(F)$. This
leads to
\begin{align}
\label{eq:alphaPositive}
\alpha^2 = \frac{4d+2}{2r-1} > 0.
\end{align}
The next step is to rule out the cases $r \geq 2$ and $r \leq -1$. If $r \geq
2$, then $\rk(G) \leq -1$. By exchanging the roles of $F$ and $G$ in the
following argument, it is enough to deal with the situation $r \leq -1$. In that
case we use (\ref{eq:alphaPositive}) and the Bogomolov inequality to
get the contradiction $2rd \leq 1$, $d < -\tfrac{1}{2}$ and $r \leq -1$.

Therefore, we know $r=0$ or $r=1$. By again interchanging the roles of $F$ and
$G$ if necessary, we only have to handle the case $r=1$. Equation
(\ref{eq:alphaPositive}) implies $d > - \tfrac{1}{2}$. By Lemma
\ref{lem:idealSheafRestriction} we get $d - 1/2 \in \Z_{\leq 0}$. Therefore, we
are left with the case in the claim.
\end{proof}
\end{lem}

\begin{proof}[Proof of Theorem \ref{thm:tiltStabilityTwistedCubics}]
Since we are only dealing with $\beta < 0$, the structure theorem for walls
in tilt stability (Theorem \ref{thm:structure_theorem}) implies that all walls intersect the left branch of the
hyperbola. In Theorem \ref{thm:smallestWall} we already determined the
smallest wall in much more generality. It intersects the
$\beta$-axis at $\beta = -3$ and $\beta = -2$. Therefore, all other walls
intersecting this branch of the hyperbola have to intersect the ray $\beta
= -2$. By Lemma \ref{lem:chAtMinusTwo} there is at most one wall on this
ray. It corresponds to the solution claimed to exist.

Let $0 \to F \to E \to G \to 0$ define a wall in $\Coh^{-2}(\P^3)$ with
$\ch(E) = (1,0,-3,5)$. One can compute $\ch^{-2}(E) = (1, 2, -1, \tfrac{1}{3})$. 
Up to interchanging the roles of $F$ and $G$ we have $\ch^{-2}(F) = (1, 1,
1/2, e)$ and $\ch^{-2}(G) = (0, 1, -3/2, 1/3 - e)$. By Lemma
\ref{lem:idealSheafRestriction} we get $F \cong \II_Z(-1)$ where $Z \in \P^3$
is a zero dimensional sheaf of length $1/6-e$ in $\P^3$. In particular, the
inequality $e \leq 1/6$ holds. The same lemma also implies that $G \cong
I_{Z'/V}(-3)$ where $Z'$ is a dimension zero subscheme of length $e + 5/6$ in
$V$. In particular, $e \geq -5/6$. Therefore, the two cases $e = \tfrac{1}{6}$
and $e = -\tfrac{5}{6}$ remain, and correspond exactly to the two sets of objects
in the Theorem.
\end{proof}


\section{Connecting Bridgeland Stability and Tilt Stability}
\label{sec:connection}

In the example of twisted cubics in the last section, we saw that the biggest
wall was defined by two different types of exact sequences. Their difference was
purely determined in codimension three. It is not very surprising
that codimension three geometry cannot be properly captured by tilt stability,
since its definition does not include the third Chern character. It seems
difficult to precisely determine how the corresponding sets of stable objects
change at this complicated wall. We will show a general way to handle this issue
by using Bridgeland stability conditions. The problem stems from the fact that
Lemma \ref{lem:wall_crossing} is in general incorrect in tilt stability.
We will see how these multiple walls in
tilt stability have to separate in Bridgeland stability in the next section for
some examples.

Let $v = (v_0, v_1, v_2, v_3)$ be the Chern character of an object in $D^b(X)$.
For any $\alpha > 0$, $\beta \in \R$, and $s > 0$ we denote the set of
$\lambda_{\alpha, \beta, s}$-semistable objects with Chern character $\pm v$ by
$M_{\alpha, \beta, s}(v)$. Analogous to our notation for twisted Chern
characters we write $v^{\beta} = (v^{\beta}_0, v^{\beta}_1, v^{\beta}_2,
v^{\beta}_3) := v \cdot e^{-\beta H}$. We also write
$$P_v := \{(\alpha, \beta) \in \R_{> 0} \times \R : \nu_{\alpha, \beta}(v) >
0 \}.$$
The goal of this section is to prove the
following theorem. Under some hypotheses, it roughly says that on one side of
the hyperbola $\{ \nu_{\alpha, \beta}(v) = 0 \}$ all the chambers and wall
crossings of tilt stability occur in a potentially refined way in Bridgeland
stability. In general, the difference between these wall crossings and the
corresponding situation in tilt stability is comparable to the difference
between slope stability and Gieseker stability. Using the theory of polynomial
stability conditions from \cite{Bay09} one can define an analogue of that
situation to make this precise. We will not do this, as we are not aware of any
interesting examples in which the difference matters.

\begin{thm}
\label{thm:wall_intersecting_hyperbola}
Let $v$ be the Chern character of an object in $D^b(X)$, $\alpha_0 > 0$,
$\beta_0 \in \R$, and $s > 0$ such that $\nu_{\alpha_0, \beta_0}(v) = 0$, $H^2
v^{\beta_0}_1 > 0$, and $Q^{\tilt}(v) \geq 0$.
\begin{enumerate}
  \item Assume there is an actual wall in Bridgeland stability for $v$ at
  $(\alpha_0, \beta_0, s)$ given by
  $$0 \to F \to E \to G \to 0.$$
  That means $\lambda_{\alpha_0, \beta_0, s}(F) = \lambda_{\alpha_0, \beta_0,
  s}(G)$ and $\ch(E) = \pm v$ for semistable $E,F,G \in \AA^{\alpha_0,
  \beta_0}(X)$. Further assume there is a neighborhood $U$ of $(\alpha_0,
  \beta_0)$ such that the same sequence also defines an actual wall in $U \cap
  P_v$, i.e., $E,F,G$ remain semistable in $U \cap P_v \cap \{ \lambda_{\alpha,
  \beta, s}(F) = \lambda_{\alpha, \beta, s}(G)\}$. Then $E[-1]$, $F[-1]$,
  $G[-1] \in \Coh^{\beta_0}(X)$ are $\nu_{\alpha_0, \beta_0}$-semistable. In
  particular, there is an actual wall in tilt stability at $(\alpha_0, \beta_0)$.
  \item Assume that all $\nu_{\alpha_0, \beta_0}$-semistable objects of class $v$ are $\nu_{\alpha_0, \beta_0}$-stable.
  Then there is a neighborhood $U$ of $(\alpha_0, \beta_0)$ such that
  $$M_{\alpha, \beta, s}(v) = M^{\tilt}_{\alpha, \beta}(v)$$
  for all $(\alpha, \beta) \in U \cap P_v$. Moreover, in this case all objects
  in $M_{\alpha, \beta, s}(v)$ are $\lambda_{\alpha, \beta, s}$-stable.
  \item Assume there is a wall in tilt stability intersecting $(\alpha_0,
  \beta_0)$. If the set of tilt stable objects is different on the two sides of
  the wall, then there is at least one actual wall in Bridgeland stability in
  $P_v$ that has $(\alpha_0, \beta_0)$ as a limiting point.
  \item Assume there is an actual wall in tilt stability for $v$ at $(\alpha_0,
  \beta_0)$ given by
  $$0 \to F^{\oplus n} \to E \to G^{\oplus n} \to 0$$
  such that $F, G \in \Coh^{\beta_0}(X)$ are $\nu_{\alpha_0, \beta_0}$-stable
  objects, $\ch(E) = v$ and $\nu_{\alpha_0, \beta_0}(F) = \nu_{\alpha_0,
  \beta_0}(G)$. Assume further that the set
  $$P_v \cap P_{\ch(F)} \cap P_{\ch(G)} \cap \{ \lambda_{\alpha, \beta, s}(F) =
  \lambda_{\alpha, \beta, s}(G)\}$$
  is non empty. Then there is a neighborhood $U$ of $(\alpha_0, \beta_0)$ such
  that $F,G$ are $\lambda_{\alpha, \beta, s}$-stable for all $(\alpha, \beta)
  \in U \cap P_v \cap \{ \lambda_{\alpha, \beta, s}(F) = \lambda_{\alpha,
  \beta, s}(G)\}$. In particular, there is an actual wall in
  Bridgeland stability restricted to $U \cap P_v$ defined by the same sequence.
\end{enumerate}
\end{thm}

Before we can prove this theorem, we need three preparatory lemmas. The
following lemma shows how to descend tilt stability on the hyperbola $\{
\nu_{\alpha, \beta}(v) = 0 \}$ to Bridgeland stability on one side of the
hyperbola. The main issue is that the hyperbola can potentially be a wall
itself.

\begin{lem}
\label{lem:stable_near_hyperbola}
Assume $E \in \Coh^{\beta_0}(X)$ is a $\nu_{\alpha_0, \beta_0}$-stable object
such that $\nu_{\alpha_0, \beta_0}(E) = 0$ and fix some $s > 0$. Then
$E[1]$ is $\lambda_{\alpha_0, \beta_0, s}$-semistable. Moreover, there is a
neighborhood $U$ of $(\alpha_0, \beta_0)$ such that $E$ is $\lambda_{\alpha,
\beta, s}$-stable for all $(\alpha, \beta) \in U \cap P_{\ch(E)}$.
\begin{proof}
By definition $E[1] \in \AA^{\alpha_0, \beta_0}(X)$. Since $\lambda_{\alpha_0, \beta_0, s}(E[1]) = \infty$, the object $E[1]$ is semistable at this point. By Theorem \ref{thm:locally_finite}, there is a locally finite wall and chamber structure such that the Harder-Narasimhan filtration of $E$ is constant in each chamber. Therefore, we can choose a neighborhood $U$ around $(\alpha_0, \beta_0)$ such that any destabilizing stable quotient $E \onto G$ in $U \cap P_{\ch(E)}$ becomes a stable quotient in the Jordan-H\"older filtration of $E[1]$ at $(\alpha_0, \beta_0, s)$.

If $G$ is supported in dimension $0$, then it could not be a destabilizing quotient anywhere. Therefore, we get $\nu_{\alpha_0, \beta_0}(G) = 0$. Let $F$, respectively $F[1]$ be the kernel of the morphism. If $F$ is supported in dimension $0$, then it would be impossible for both $F \in \AA^{\alpha, \beta}(X)$ whenever $(\alpha, \beta) \in U \cap P_{\ch(E)}$ and $F[1] \in \AA^{\alpha_0, \beta_0}(X)$. In particular, we also get $\nu_{\alpha_0, \beta_0}(F) = 0$ or $F = 0$.

The long exact sequence with respect to $\Coh^{\beta_0}(X)$ is given by
$$0 \to \HH^{-1}_{\beta_0}(F[1]) \to E \to \HH^{-1}_{\beta_0}(G[1]) = G \to
\HH^0_{\beta_0}(F[1]) \to 0.$$
Due to Lemma \ref{lem:large_volume_limit_bridgeland}, the object $\HH_{\beta_0}^{0}(F[1])$ is supported in dimension $0$, but all of $F$ is not. Since $E$ is $\nu_{\alpha_0, \beta_0}$-stable, we must have $F = 0$. Therefore, $E=G$ is stable.
\end{proof}
\end{lem}

At the hyperbola the Chern character of stable objects usually changes between
$v$ and $-v$. This comes hand in hand with objects leaving the heart while a
shift of the object enters the heart. The next lemma deals with the question
which shift is at which point in the category.

\begin{lem}
\label{lem:which_shift}
Let $v$ be the Chern character of an object in $D^b(X)$, $\alpha_0 > 0$,
$\beta_0 \in \R$, and $s > 0$ such that $\nu_{\alpha_0, \beta_0}(v) = 0$, $H^2 v^{\beta}_1 > 0$, and $Q^{\tilt}(v) \geq 0$. Assume there is a path $\gamma: [0,1] \to \overline{P_v}$ with
$\gamma(1) = (\alpha_0, \beta_0)$, $\gamma([0,1)) \subset P_v$, $E \in
\AA^{\gamma(t)}(X)$ is $\lambda_{\gamma(t), s}$-semistable for all $t \in [0,1)$,
and $\ch(E) = v$. Then $E[1] \in \AA^{\alpha_0, \beta_0}(X)$.
\begin{proof}
The map $[0,1] \to \R$, $t \mapsto \phi_{\gamma(t),s}(E)$ is continuous.
Thus, there is $m \in \{0,1\}$ such that $E[m] \in \AA^{\alpha_0, \beta_0}(X)$
is $\lambda_{\alpha_0, \beta_0, s}$-semistable. Since $\lambda_{\alpha_0, \beta_0, s}(E[m]) = \infty$ for any $s$, the $\lambda_{\alpha_0, \beta_0, s}$-semistability of $E[m]$ is independent of $s$. Assume $m = 0$. Then Lemma \ref{lem:large_volume_limit_bridgeland} implies that $\HH^{-1}_{\beta_0}(E)$ is
$\nu_{\alpha_0, \beta_0}$-semistable and $\HH^0_{\beta_0}(E)$ is a sheaf
supported in dimension $0$. This implies $H^2 \ch^{\beta_0}_1(E) \leq 0$.
Therefore, $H^2 v_1^{\beta_0} > 0$ implies $\ch(E) = -v$. This leads to
$$\Im Z_{\gamma(t),s}(E) = - \Im Z_{\gamma(t),s}(v) < 0$$
for all $t \in [0,1)$ in contradiction to $E \in \AA^{\alpha_0, \beta_0}(X)$.
\end{proof}
\end{lem}

The final lemma restricts the possibilities for semistable objects that leave
the heart while a shift enters the heart.

\begin{lem}
\label{lem:shift_enters}
Let $\gamma: [0,1] \to \R_{> 0} \times \R$ be a path, $\gamma(1) =
(\alpha_0, \beta_0)$, $s > 0$, $E \in D^b(X)$ be an object such that $E \in
\AA^{\gamma(t)}(X)$ is $\lambda_{\gamma(t),s}$-semistable for all $t \in[0,1)$,
and $E[1] \in \AA^{\alpha_0, \beta_0}(X)$ is $\lambda_{\alpha_0, \beta_0,
s}$-semistable. Then $E \in \Coh^{\beta_0}(X)$ is $\nu_{\alpha_0,
\beta_0}$-semistable.
\begin{proof}
The continuity of $[0,1] \to \R$, $t \mapsto \phi_{\gamma(t),s}(E)$ implies
$\Im Z_{\alpha_0, \beta_0, s}(E) = 0$. Then Lemma
\ref{lem:large_volume_limit_bridgeland} implies that $\HH^{-1}_{\beta_0}(E[1])$
is $\nu_{\alpha_0, \beta_0}$-semistable and $\HH^0_{\beta_0}(E[1])$ is a sheaf
supported in dimension $0$. In particular, there is a non trivial map $E[1] \to
\HH^0_{\beta_0}(E[1])$ unless $\HH^0_{\beta_0}(E[1]) = 0$. Since $E \in
\AA^{\gamma(t)}(X)$ for $t \in [0,1)$ one obtains
$$\phi_{\gamma(t), s}(E[1]) > 1 = \phi_{\gamma(t), s}(\HH^0_{\beta_0}(E[1])).$$
The semi-stability of $E$ implies $\HH^0_{\beta_0}(E[1]) = 0$.
\end{proof}
\end{lem}

Together with these three lemmas, we can prove the Theorem.

\begin{proof}[Proof of Theorem \ref{thm:wall_intersecting_hyperbola}]
We start by proving (1). Since $0 \to F \to E \to G \to 0$ also defines a wall
in $U \cap P_v$ we know there is $m \in \Z$ such that $E[m]$, $F[m]$, $G[m] \in
\AA^{\alpha, \beta}(X)$ for $(\alpha, \beta) \in U \cap P_v$. By Lemma
\ref{lem:which_shift} this implies $m = -1$ and Lemma \ref{lem:shift_enters}
shows $E[-1]$, $F[-1]$ and $G[-1]$ are all $\nu_{\alpha_0, \beta_0}$-semistable.

This defines a wall in tilt stability unless $\nu_{\alpha, \beta}(F) =
\nu_{\alpha, \beta}(G)$ for all $(\alpha, \beta) \in \R_{> 0} \times \R$. But
this is only possible if $\lambda_{\alpha, \beta, s}(F) = \lambda_{\alpha,
\beta, s}(G)$ is equivalent to $\nu_{\alpha, \beta}(v) = 0$.

We continue by showing part (2). By assumption $(\alpha_0, \beta_0)$ does not
lie on any wall for $v$ in tilt stability. Let $U'$ be a neighborhood of $(\alpha_0,
\beta_0)$ that does not intersect any such wall. In particular, this means
$M^{\tilt}_{\alpha, \beta}(v)$ is constant on $U'$. By part (i) any wall in
Bridgeland stability that intersects the hyperbola $\{ \nu_{\alpha, \beta}(v) =
0 \}$ and stays an actual wall in some part of $P_v$ comes from a wall in
tilt stability. Therefore, we can choose a neighborhood $U''$ of $(\alpha_0,
\beta_0)$ such that there is no wall in Bridgeland stability for $v$ in $U'' \cap
P_v$. We define $U:= U' \cap U''$ and choose $(\alpha, \beta) \in U$.

The inclusion $M^{\tilt}_{\alpha, \beta}(v) \subset M_{\alpha, \beta, s}(v)$ is
a restatement of Lemma \ref{lem:stable_near_hyperbola}. Let $E \in M_{\alpha,
\beta, s}(v)$. There is $m \in \Z$ such that $E[m] \in \AA^{\alpha_0, \beta_0}$
is a $\lambda_{\alpha_0, \beta_0, s}$-semistable object. By Lemma
\ref{lem:which_shift} one gets $m = 1$ and Lemma \ref{lem:shift_enters} implies
$E \in \Coh^{\beta}(X)$ is tilt semistable, i.e., $E \in M^{\tilt}_{\alpha,
\beta}(v)$.

Part (3) follows from (2), while (4) is an immediate application of Lemma
\ref{lem:stable_near_hyperbola}.
\end{proof}


\section{Examples in Bridgeland Stability}
\label{sec:exmaples_bridgeland}

In this section the techniques for connecting Bridgeland stability and
tilt stability are applied to the previous examples on $\P^3$.

\subsection{Certain Sheaves}

Fix $s > 0$. Recall that $m,n \in \Z$ are integers with $n < m$ and $i,j \in \N$
are positive integers. There is a class defined by $v = i \ch(\OO(m)) - j
\ch(\OO(n))$. We will show that there is a path close to one branch of the
hyperbola $\{ \Im Z_{\alpha, \beta, s}(v) = 0 \}$, where the first
wall crossing described in Theorem \ref{thm:smallestWall} happens in Bridgeland
stability. The first moduli space after this wall turns out to be
smooth and irreducible. Moreover, at the end of the path stable objects
are exactly slope stable sheaves with Chern character $v$.

\begin{thm}
\label{thm:first_wall_general}
Assume that $(v_0, v_1, v_2) \in \Z \oplus \Z \oplus \tfrac{1}{2} \Z$ is a primitive vector. There is a path $\gamma:
[0,1] \to \R_{>0} \times \R \subset \Stab(\P^3)$ that satisfies the following
properties.
\begin{enumerate}
  \item The first wall for objects of class $v$ along $\gamma$ is given by $\lambda_{\alpha, \beta, s}
  (\OO(m)) = \lambda_{\alpha, \beta, s} (\OO(n))$. At $\gamma(0)$ there are no
  semistable objects. After the wall, the moduli space is smooth, irreducible,
  and projective.
  \item If $i \geq j$, then at $\gamma(1)$ the semistable objects are exactly
  slope stable coherent sheaves $E$ with $\ch(E) = v$. Moreover, there are no
  strictly semistable objects.
\end{enumerate}
\end{thm}

The theorem is stated for arbitrary $i$, $j$, $m$, $n$, but only interesting under restrictions. In the proof, we will show that the moduli space after the first wall is a moduli space of quiver representation on a generalized Kronecker quiver $Q$ which can very well be empty.

It is possible to put precise numerical restrictions on $i$, $j$, $m$, and $n$ under which the space is non-empty. By \cite[Proposition 4.4]{Kin94} the space is non-empty if and only $(i,j)$ is a Schur root for $Q$.
The description of the Schur roots of a generalized Kronecker quiver can for example be found in \cite[Example 7]{Fae13}.

\begin{proof}[Proof of Theorem \ref{thm:first_wall_general}]
By Theorem \ref{thm:smallestWall} there is a wall in tilt stability defined by the equation
$\nu_{\alpha, \beta} (\OO(m)) = \nu_{\alpha, \beta} (\OO(n))$. Moreover, there
is no smaller wall. Since $(v_0, v_1, v_2)$ is a primitive
vector, any moduli space of $\nu_{\alpha, \beta}$-semistable objects for $v$,
such that $(\alpha, \beta)$ does not lie on a wall, consists solely of
tilt stable objects. Let $Y \subset \{ \Im Z_{\alpha, \beta, s}(v) = 0 \}$ be
the branch of the hyperbola that intersects this wall. Due to Theorem
\ref{thm:wall_intersecting_hyperbola} we can find a path $\gamma: [0,1] \to
\R_{>0} \times \R \into \Stab(\P^3)$ close enough to $Y$ such that all moduli
spaces of tilt stable objects that occur on $Y$ outside of any wall are moduli
spaces of Bridgeland stable objects along $\gamma$. Moreover, we can assume that
$\gamma$ intersects no wall twice and the first wall crossing is given by
$\lambda_{\alpha, \beta, s} (\OO(m)) = \lambda_{\alpha, \beta, s} (\OO(n))$.

Part (2) can be proven as follows. By the choice of $\gamma$, we have
$M^{\tilt}_{\gamma(1)}(v) = M_{\gamma(1), s}(v)$. In tilt stability $\gamma(1)$
is above the largest wall. Therefore, Lemma \ref{lem:large_volume_limit_tilt}
and Lemma \ref{lem:stable_sheaves_stable_in_tilt} imply that
$M^{\tilt}_{\gamma(1)}(v)$ consists of slope stable sheaves $E$ with $\ch(E)
= v$.

We will finish the proof of (1) by showing that we get a
moduli space of representations on a generalized Kronecker quiver. Let $t \in (0,1)$ be such that 
$M_{\gamma(t), s}(v)$ is the first non-empty moduli space on $\gamma$. Let $Q$ be the generalized Kronecker quiver with $N = \dim \Hom(\OO(n), \OO(m))$ arrows. \\

\centerline{
\xygraph{
!{<0cm,0cm>;<1cm,0cm>:<0cm,1cm>::}
!{(0,0) }*+{\circ}="1"
!{(5,0) }*+{\circ}="2"
!{(2.5,0.1) }*+{\vdots \ N}="dots"
"1"-@{>}@/_0.4cm/"2"
"1"-@{>}@/^0.4cm/"2"
}} \ \\

For any representation $V$ of $Q$ we denote the dimension vector by
$\underline{\dim}(V)$. If $\theta: \Z \oplus \Z \to \Z$ is a homomorphism with
$\theta(j,i) = 0$ we say that a representation $V$ of $Q$ with
$\underline{\dim}(V) = (j, i)$ is $\theta$-(semi)stable if for any
subrepresentation $W \into V$ the inequality $\theta(W) > (\geq) 0$ holds.

Due to \cite{Kin94} there is a projective coarse moduli space $K_{\theta}$ that represents stable
complex representations with dimension vector $(j,i)$. If there are no strictly
semistable representation, then $K_{\theta}$ is a fine moduli space. By Theorem \ref{thm:smallestWall}
the moduli space solely consists of extensions of $\OO(n)^{\oplus j}[1]$ and $\OO(m)^{\oplus i}$. Therefore,
we can find $\theta_t$ such that $\theta_t$-stability and
Bridgeland stability at $\gamma(t)$ match. More precisely, there is a bijection
between Bridgeland stable objects at $\gamma(t)$ with Chern character $v$ and
$\theta_t$-stable complex representations with dimension vector $(j,i)$. We denote
this specific moduli space of quiver representations by $K$. Since the quiver
has no relation and $i$, $j$ have to be coprime, we get that $K$ is a smooth
projective variety.

We want to construct an isomorphism between $K$ and the moduli space
$M_{\gamma(t), s}(v)$ of Bridgeland stable complexes with Chern character $v$.
In order to do so, we need to make the above bijection more precise. Let
$\Hom(O(n), O(m)) = \bigoplus_l \C \varphi_l$. There is a functor $\FF: \Rep(Q)
\to D^b(\P^3)$ that sends a representation whose maps are given by $i \times j$ matrices $A_l$ for $l = 1, \ldots, N$ to the two term complex $\OO(n)^{\oplus j} \to \OO(m)^{\oplus i}$ with morphism 
\[
\begin{pmatrix}
s_1 \\
\vdots \\
s_j
\end{pmatrix}
\mapsto \sum_l A_l \begin{pmatrix}
\varphi_l(s_1) \\
\vdots \\
\varphi_l(s_j)
\end{pmatrix}. 
\]
This functor induces the bijection between stable objects mentioned above.

Let $S$ be a scheme over $\C$. A representation of $Q$ over $S$ is given by $N$
maps $f_1, \ldots, f_N: V \to W$ for locally free sheaves $V, W \in \Coh(S)$.
The functor above can be generalized to the relative setting as $\FF_S:
\Rep_S(Q) \to D^b(\P^3 \times S)$ sending a family of representations $f_l: V \to W$, $l = 1, \ldots, N$ to the two term
complex $V \boxtimes \OO(n) \to W \boxtimes \OO(m)$ where the map is given by
$\sum v \otimes s \mapsto \sum \sum_l f_l(v) \otimes \varphi_l(s)$.

If $\EE$ is a family of $\theta_t$-semistable representation over $S$, then
we get $\FF(\EE_s) = \FF_S(\EE)_s$ for any $s \in S$. That induces a bijective
morphism from $K$ to $M_{\gamma(t), s}(v)$. We want to show that this morphism
is in fact an isomorphism. In order to so, we will first need to prove
smoothness of the moduli space of Bridgeland stable objects at $\gamma(t)$.

We have $\dim M_{\gamma(t), s}(v) = \dim K = jiN - i^2 - j^2 + 1$. For any $E
\in M_{\gamma(t), s}(v)$ the Zariski tangent space at $E$ is given by $\Ext^1(E,E)$
by standard deformation theory arguments (see \cite{Ina02, Lie06}). We have an exact triangle
\begin{equation}
\label{eq:quivertriangle}
O(m)^{\oplus i} \to E \to O(n)^{\oplus j}[1].
\end{equation}
Since $E$ is stable we have $\Hom(O(n)[1], E) = 0$. Applying $\Hom(\OO(n),
\cdot)$ to (\ref{eq:quivertriangle}) leads to $\Hom(O(n), E) = \C^{Ni - j}$. The
same way we get $\Hom(O(m), E) = \C^i$ and $\Ext^1(O(m), E) = 0$. Since $E$ is
stable, the equation $\Hom(E,E) = \C$ holds. Applying $\Hom(\cdot, E)$ to
(\ref{eq:quivertriangle}) leads to the following long exact sequence.
$$0 \to \C \to \C^{i^2} \to \C^{Nij - j^2} \to \Ext^1(E,E) \to 0.$$
That means $\dim \Ext^1(E,E) = Nij - j^2 - i^2 + 1 = \dim M_{\gamma(t), s}(v)$,
i.e., $M_{\gamma(t), s}(v)$ is smooth.

Since there are no strictly semistable objects, we can use the main result of
\cite{PT15} to infer that $M_{\gamma(t), s}(v)$ is a smooth proper algebraic
space of finite type over $\C$. According to \cite[Page 23]{Knu71} there is a
fully faithful functor from smooth proper algebraic spaces of finite type over
$\C$ to complex manifolds. Since any bijective holomorphic map between two
complex manifolds has a holomorphic inverse we are done.
\end{proof}

\subsection{Twisted Cubics}

In the example of twisted cubic curves, we described all walls in
tilt stability for $\beta < 0$ in Theorem
\ref{thm:tiltStabilityTwistedCubics}. We will translate this result into
Bridgeland stability via Theorem \ref{thm:wall_intersecting_hyperbola}.

\begin{figure}[h!]
        \centering
        \includegraphics[width=0.7\textwidth]{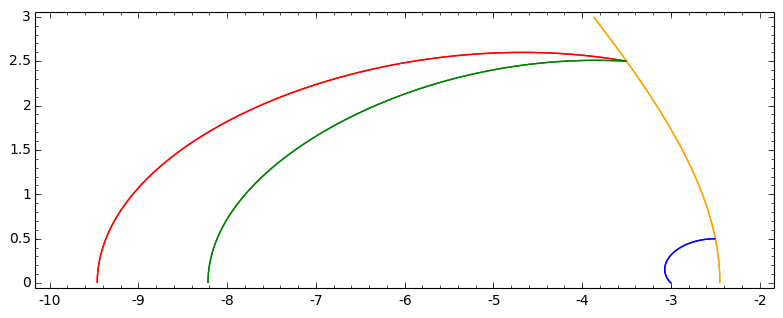}
        \caption{Walls in Bridgeland
        stability}\label{fig:bridgeland_walls_twisted}
\end{figure}

\begin{thm}
\label{thm:bridgeland_twisted_cubics}
There is a path $\gamma: [0,1] \to \R_{>0} \times \R \subset \Stab(\P^3)$ that
crosses the following walls for $v = (1, 0, -3, 5)$ in the following order. The
walls are defined by the two given objects having the same slope. Moreover, all
strictly semistable objects at each of the walls are extensions of those two
objects. Let $V$ be a plane in $\P^3$, $P \in \P^3$
and $Q \in V$.
\begin{enumerate}
  \item $\OO(-2)^{\oplus 3}$, $\OO(-3)[1]^{\oplus 2}$
  \item $\II_P(-1)$, $\OO_V(-3)$
  \item $\OO(-1)$, $\II_{Q/V}(-3)$
\end{enumerate}
The chambers separated by those walls lead to the following moduli spaces.
\begin{enumerate}
  \item The empty space $M_0 = \emptyset$.
  \item A smooth projective variety $M_1$.
  \item A space with two components $M_2 \cup M'_2$. The space $M_2$ is a blow up of $M_1$ in the incidence variety parametrizing a point in a plane in $\P^3$. The second component $M'_2$ is a $\P^9$-bundle over the smooth variety $\P^3 \times (\P^3)^{\vee}$ parametrizing pairs $(\II_P(-1)$, $\OO_V(-3))$. The two
  components intersect transversely in the exceptional locus of the blow up.
  \item The Hilbert scheme of curves $C$ with $\ch(\II_C) = (1, 0, -3, 5)$. It
  is given as $M_2 \cup M'_3$ where $M'_3$ is a blow up of $M'_2$ in the smooth
  locus parametrizing objects $\II_{Q/V}(-3)$.
\end{enumerate}
\end{thm}

Before proceeding with the proof, we would like to make a few remarks on Figure \ref{fig:bridgeland_walls_twisted}. While the precise picture depends on $s$, the relative position of the given walls does not. Note that the theorem describes the moduli spaces along some path to the left of the hyperbola. However, it does not show that there are no further walls away from the hyperbola. For example, it is possible that the wall crossing for fixed small $\alpha$ and varying $\beta$ is different than expected from the figure. In general, it is completely open how walls behave away from the hyperbola for both this and other examples.

\begin{proof}[Proof of Theorem \ref{thm:bridgeland_twisted_cubics}]
Let $\gamma$ be the path that exists due to Theorem
\ref{thm:first_wall_general}. The fact that all the walls on this path occur in
this form is a direct consequence of Theorem
\ref{thm:wall_intersecting_hyperbola} and Theorem
\ref{thm:tiltStabilityTwistedCubics}.

By Theorem \ref{thm:first_wall_general} we know that $M_0 = \emptyset$, that $M_1$ is smooth, projective, and irreducible and that the Hilbert scheme occurs at the end of the path. The main result in \cite{PS85} is that this Hilbert scheme has exactly two smooth irreducible components of dimension $12$ and $15$ that intersect transversely in a locus of dimension $11$. The $12$-dimensional component $M_2$ contains the space of twisted cubics as an open subset. The $15$-dimensional component $M'_3$ parametrizes plane cubic curves with a potentially but no necessarily embedded point. Moreover, the intersection parametrizes plane singular cubic curves with a spatial embedded point at a singularity. In particular, those curves are not scheme theoretically contained in a plane.

Strictly semistable objects at the biggest wall are given by extensions of $\OO(-1)$, $\II_{Q/V}(-3)$. For an ideal sheaf of a curve this can only mean that there is an exact sequence
$$0 \to \OO(-1) \to I_C \to \II_{Q/V}(-3) \to 0.$$
This can only exist, if $C \subset V$ scheme theoretically. Therefore, the last wall only modifies the second component. The moduli space of objects $\II_{Q/V}(-3)$ is the incidence variety of points in a plane inside $\P^3 \times (\P^3)^{\vee}$. In particular, it is smooth and of dimension $5$. A straightforward computation shows $\Ext^1(\OO(-1), \II_{Q/V}(-3)) = \C$. That means at this wall the irreducible locus of extensions $\Ext^1(\II_{Q/V}(-3), \OO(-1)) = \C^{10}$ is contracted onto a smooth locus. Moreover, for each sheaf $\II_{Q/V}(-3)$ the fiber is given by $\P^9$. This means the contracted locus is a divisor. By a classical result of Moishezon \cite{Moi67} any proper birational morphism $f: X \to Y$ between smooth projective varieties such that the contracted locus $E$ is irreducible and the image $f(E)$ is smooth is the blow up of $Y$ in $f(E)$. Therefore, to see that $M'_3$ is the blow up of $M'_2$ we need to show that $M'_2$ is smooth.

At the second wall strictly semistable objects are given by extensions of $\II_P(-1)$ and $\OO_V(-3)$. One computes $\Ext^1(\II_P(-1), \OO_V(-3)) = \C$ for $P \in V$, $\Ext^1(\II_P(-1), \OO_V(-3)) = 0$ for $P \notin V$, and $\Ext^1(\OO_V(-3), \II_P(-1)) = \C^{10}$. The objects $\II_P(-1)$ and $\OO_V(-3)$ vary in $\P^3$ respectively $(\P^3)^{\vee}$ that are both fine moduli spaces. Therefore, the component $M'_2$ is a $\P^9$-bundle over the moduli space of pairs $(\OO_V(-3), \II_P(-1))$, i.e., $\P^3 \times (\P^3)^{\vee}$. This means $M'_2$ is smooth and projective.

We are left to show that $M_2$ is the blow up of $M_1$. We already know that $M_2$ is the smooth component of the Hilbert scheme containing twisted cubic curves. Moreover, $M_1$ is smooth by Theorem \ref{thm:first_wall_general}. We want to apply the above result of Moishezon again. The exceptional locus of the map from $M_2$ to $M_1$ is given by the intersection of the two components in the Hilbert scheme. By \cite{PS85} this is an irreducible divisor in $M_2$. Due to $\Ext^1(\II_P(-1), \OO_V(-3)) = \C$ for $P \in V$ the image is as predicted.
\end{proof}



\section*{Appendix: Computing Walls Algorithmically}

The computational side for determining walls in tilt-stability in this article is rather straightforward. In this section we discuss how this problem can be solved by computer calculations. The proof of the following Lemma provides useful techniques for actually determining walls. As before $X$ is a smooth projective threefold, $H$ an ample polarization, and for any $\alpha > 0$, $\beta \in \R$ we have a very weak stability condition $(\Coh^{\beta}(X), Z^{\tilt}_{\alpha, \beta})$.

\begin{lem}
\label{lem:rational_beta}
Let $\beta \in \Q$ and $v$ be the Chern character of some object of $D^b(X)$.
Then there are only finitely many walls in tilt stability for this fixed $\beta$
with respect to $v$.
\begin{proof}
Any wall has to come from an exact sequence $0 \to F \to E \to G \to 0$ in
$\Coh^{\beta}(X)$. Let $H \cdot \ch^{\beta}_{\leq 2}(E) = (R, C, D)$ and
$H \cdot \ch^{\beta}_{\leq 2}(F) = (r, c, d)$. Notice that due to the fact $\beta
\in \Q$ the possible values of $r$, $c$ and $d$ are discrete in $\R$.
Therefore, it will be enough to bound those values to get finiteness.

By the definition of $\Coh^{\beta}(X)$ one has $0 \leq c \leq C$. If $C = 0$,
then $c = 0$ and we are dealing with the unique vertical wall. Therefore, we may
assume $C \neq 0$. Let $\Delta := C^2 - 2RD$. The Bogomolov inequality together
with Lemma \ref{lem:discriminant_properties} implies $0 \leq c^2 - 2rd \leq
\Delta$. Therefore, we get
$$\frac{c^2}{2} \geq rd \geq \frac{c^2 - \Delta}{2}.$$
Since the possible values of $r$ and $d$ are discrete in $\R$, there are
finitely many possible values unless $r = 0$ or $d = 0$.

Assume $R = r = 0$. Then the equality $\nu_{\alpha, \beta}(F) = \nu_{\alpha,
\beta}(E)$ holds if and only if $Cd-Dc = 0$. In particular, it is independent of
$(\alpha, \beta)$. Therefore, the sequence does not define a wall.

If $r = 0$, $R \neq 0$, and $D - d \neq 0$, then using the same type of inequality for $G$ instead of $E$ will finish the proof. If $r = 0$ and $D - d = 0$, then $d = D$ and there are are only finitely many walls like this, because we already bounded $c$.

Assume $D = d = 0$. Then the equality $\nu_{\alpha, \beta}(F) = \nu_{\alpha,
\beta}(E)$ holds if and only if $Rc - Cr = 0$. Again this cannot define a wall.

If $d = 0$, $D \neq 0$, and $R - r \neq 0$, then using the same type of inequality for $G$ instead of $E$ will finish the proof. If $d = 0$ and $R - r = 0$, then $r = R$ and there are are only finitely many walls like this, because we already bounded $c$.
\end{proof}
\end{lem}

Note that together with the structure theorem for walls in tilt stability this
lemma implies that there is a biggest semicircle on each side of the vertical
wall.

The proof of the Lemma tells us how to algorithmically solve the problem of determining all walls on a given vertical line. Assuming that $\beta$ does not give the unique vertical wall, we have the following inequalities for any exact sequence $0 \to F \to E \to G \to 0$ defining a potential wall.
\begin{align*}
0 < H \cdot \ch_1^{\beta} (F) &< H \cdot \ch_1^{\beta} (E), \\
0 < H \cdot \ch_1^{\beta} (G) &< H \cdot \ch_1^{\beta} (E), \\
Q^{\tilt}(F,F) &\geq 0, \\
Q^{\tilt}(G,G) &\geq 0, \\
Q^{\tilt}(E,F) &\geq 0.
\end{align*}

Moreover, we need $H \cdot \ch(F)$ and $H \cdot \ch(G)$ to be in the lattice spanned by Chern characters of objects in $D^b(X)$. Finally, the fact that the Chern classes of $F$ and $G$ are integers puts further restrictions on the possible values of the Chern characters. The code for a concrete implementation in the case of $\P^3$ in \cite{Sage} can be found at

\url{https://sites.google.com/site/benjaminschmidtmath/}.

We computed the previous example of twisted cubics with it and obtained the same walls as above. Similar computations for the case of elliptic quartic curves will occur in a future article joint with Patricio Gallardo and C\'esar Lozano Huerta.


\renewcommand{\refname}{References}

\end{document}